\documentclass[10pt]{article}
\usepackage{amsmath, amsthm,amssymb}
\usepackage{bm}
\usepackage{bbm}
\usepackage{xcolor}
\usepackage[hidelinks]{hyperref} 
\usepackage{makeidx}
\usepackage{enumerate}
\usepackage{stmaryrd}
\usepackage{tikz,color,enumitem}
\numberwithin{equation}{section}
\date{}

\newtheorem{thm}{Theorem}[section]
\newtheorem{thma}{Theorem}

\newtheorem{prop}[thm]{Proposition}
\newtheorem{cor}[thm]{Corollary}
\newtheorem{lem}[thm]{Lemma}

\newtheorem{claim}[thm]{Claim}

\newtheorem{conjecture}[thm]{Conjecture}
\newtheorem{conj}[thma]{Conjecture}

\newtheorem{defn}[thm]{Definition}

\newcommand{\ca}{\curvearrowright}
\newcommand{\G}{\Gamma}
\newcommand{\La}{\Lambda}

\newcommand{\ra}{{\rightarrow}}

\newcommand{\euu}{\mathcal{U}}
\newcommand{\paP}{\mathcal{P}}

\newcommand{\mc}{\mathbb{C}}
\newcommand{\Sg}{\Sigma}
\newcommand{\email}{Email: } 
\begin{document}

\title{Invariant subalgebras of von Neumann algebras arising from negatively curved groups}
\author{Ionu\c t Chifan, Sayan Das, Bin Sun}

\maketitle

\begin{abstract} 
\noindent Using an interplay between geometric methods in group theory and soft von Neuman algebraic techniques we prove that for any icc, acylindrically hyperbolic group $\Gamma$ its von Neumann algebra $L(\Gamma)$ satisfies the so-called ISR property: \emph{any von Neumann subalgebra $N\subseteq L(\Gamma)$ that is normalized by all group elements in $\Gamma$ is of the form $N= L(\Sigma)$ for a normal subgroup $\Sigma \lhd \Gamma$.}  In particular, this applies to all groups $\Gamma$ in each of the  following classes:  all icc (relatively) hyperbolic groups, most mapping class groups of surfaces, all outer automorphisms of free groups with at least three generators, most graph product groups arising from simple graphs without visual splitting, etc. This result answers positively an open question of Amrutam and Jiang from \cite{AJ22}.

\noindent In the second part of the paper we obtain similar results for factors associated with groups that admit nontrivial (quasi)cohomology valued into various natural representations.  In particular, we establish the ISR property for all icc, nonamenable groups that have positive first $L^2$-Betti number and contain an infinite amenable subgroup.         
 \end{abstract}

\section{Introduction}

Given a locally compact group $G$, and a continuous homomorphism of $G$ into the group of $\ast$-automorphisms of an operator algebra $\mathcal A$, say $\Phi: G \rightarrow {\rm Aut}(\mathcal A) $, one associates a triple $(\mathcal A, \Phi, G)$ called a \textit{covariant system}. The notion of a covariant system is a generalization of Murray and von Neumann's well-known \textit{group measure space} construction \cite{MvN36}. Due to its connection with the mathematical formulation of quantum field theory and statistical mechanics, covariant systems have been well studied over the years, see \cite{BR79, DKR66, Tak67, JQ91}.

\vskip 0.07in
\noindent The question of Galois correspondence for covariant group systems was initiated by Takesaki and Tatsuuma in \cite{TT71}, where they established that for a locally compact group $G$ there is a one-to-one correspondence between closed normal subgroups of $G$ and invariant von Neumann subalgebras of $L^{\infty}(G)$. 
In \cite{JQ91}, Jorgensen and Quan further developed the study of Galois correspondence for covariant group systems associated with a locally compact group $G$. They established that there is a one-to-one correspondence between $G$-invariant $\ast$-subalgebras of $L^1(G)$ and the lattice of normal subgroups of $G$, see \cite[Theorem 4.7]{JQ91}.  They also established a one-to-one correspondence between invariant Hopf $C^*$-subalgebras of $C^*_r(G)$ and the lattice of normal subgroups of $G$, see \cite[Theorem 5.10]{JQ91}.

\vskip 0.07in
\noindent A natural analogue of the aforementioned Galois correspondence results in the setting of discrete groups is to study the structure of invariant von Neumann subalgebras of the associated group von Neumann algebra. This study was undertaken by two of the authors of this paper in \cite{CD19}, where the structure of $\G$-invariant factors inside $L(\G)$, for $\G$ an icc discrete group, was investigated. In particular, it was established that \cite[Theorem 3.15]{CD19} for any $\G$-invariant $\rm II_1$ factor $N \subset L(\G)$ there exists a normal subgroup $\La \lhd \G$, with $N \subseteq L(\La) \subseteq N \vee N' \cap L(\G)$. In particular, $\G$-invariant irreducible subfactors of $L(\G)$ can only arise from normal subgroups of $\G$.
\vskip 0.07in
\noindent Around the same time, motivated by Peterson's seminal work \cite{Pet16} on noncommutative Margulis normal subgroup theorem via character rigidity and noncommutative Poisson boundaries (see also \cite{CP13, CP18, DP22}), Alekseev and Brugger \cite{AB21} studied $\G$-invariant subfactors of $L(\G)$, where  $\G$ is an irreducible lattice in a simple higher rank Lie-group, and established that any $\G$-invariant subfactor of $L(\G)$ has finite Jones index. This result was established by adapting Peterson's techniques, though as mentioned in \cite{AB21}, the result can also be derived by combining \cite[Theorem 3.15]{CD19}, with Margulis' normal subgroup Theorem. 

\vskip 0.07in
\noindent Motivated by the results in \cite{Pet16, AB21, CD19}, Kalantar and Panagopoulos \cite{KP21} showed that any $\G$-invariant von Neumann subalgebras of $L(\G)$ arise as the group von Neumann algebra of a normal subgroup, with $\G$ being an irreducible lattice in a connected semi-simple Lie group $G$ with trivial center, no nontrivial compact factors, and such that all simple factors of $G$ are higher rank. This result as well as \cite{CD19} further inspired Amrutam and Jiang to study invariant von Neumann subalgebras of $L(\G)$ for other classes of groups complementary to lattices in higher rank groups \cite{AJ22}. Specifically they established the ISR property for $L(\G)$, whenever $\G$ is a nonamenable torsion free group that is either hyperbolic, or has positive first $L^2$-Betti number and satisfies Peterson-Thom's condition ($\ast$) in \cite{PT11}. In their investigation they also raised the question whether the ISR property holds true for all factors associated to all icc hyperbolic groups or, more generally, all icc acylindrically hyperbolic groups \cite{Osi16}. These two questions were precisely the impetus of our investigation.  
 \vskip 0.07in

\noindent However, despite all these results, there are still many remarkable classes of groups for which it was not known whether their corresponding group factors satisfy the ISR property. Making use of Theorem \ref{thm: invariant factors} and the aforementioned strategy, in this paper we investigate $\Gamma$-invariant subalgebras of factors $L(\Gamma)$ arising from natural classes of ``negatively curved'' groups $\Gamma$ intensively studied in the geometric and representation theory of groups.

\vskip 0.07in
\noindent In the first part of the paper, using a combinatorial condition that relies on geometric word analysis and  $n$-qons inequalities for hyperbolically embedded subgroups from \cite{Osi07,DGO17} in conjunction with a generic von Neumann algebraic argument (see also Theorem \ref{thm: invariant factors}), we establish the ISR property for the von Neumann algebras of all icc, acylindrically hyperbolic groups, \cite{Osi16}. 

\begin{thma}\label{b} Let $\Gamma$ be any icc,  acylindrically hyperbolic group. Then for any von Neumann subalgebra $N\subseteq L(\Gamma)$ satisfying $\Gamma \subset \mathcal N_{L(\Gamma)}(N)$, one can find a normal subgroup $\Sigma \lhd \Gamma$ such that $N= L(\Sigma)$.  

\end{thma}

\noindent This result generalizes several of the prior results, (\cite[Theorem 1.2]{AJ22}, \cite[Theorem 3.15, Corollary 3.17]{CD19}) and it also answers positively the second aforementioned question of Amrutam-Jiang, (see  \cite[Question 5.1]{AJ22}). We notice that Theorem \ref{b} applies to a vast category of geometric groups including all icc hyperbolic groups, all icc relative hyperbolic groups, most mapping class groups, all outer automorphism groups of free groups with at least three generators, all nontrivial graph product groups whose underlying graph does not admit a visual splitting, etc.  To this end, we emphasize that our approach for Theorem \ref{b} is very different in essence from the deformation/rigidity theory methods employed in \cite{CD19}, using more generic von Neumann algebraic technique (Theorems \ref{thm: invariant factors} and \ref{thm:acylindricallyinvfactor}) in combination to very strong group theoretic properties of the hyperbolically embedded subgroups of these groups (Theorem \ref{8gon}). We mention in passing that since icc acylindrically hyperbolic groups $\Gamma$ do not have nontrivial normal amenable subgroups, Theorem \ref{b} yields that $L(\Gamma)$ also does not have nontrivial amenable von Neumann subalgebras that is normalized by $\Gamma$.

\vskip 0.07in

\noindent In the second part of the paper, we shift our perspective and study $\Gamma$-invariant subalgebras of $L(\Gamma)$ via methods similar to the ones used in \cite{CD19}.  As in that case, the negative curvature information we heavily exploit, is the existence of unbounded (quasi)cocycles on $\Gamma$ that are valued into its left regular representation. Even though our techniques are largely based on the analysis of (quasi)-cocycles and arrays maps developed within the deformation/rigidity theory framework \cite{Pet09,Pet09b, CP10,Si10,Va10,CS11,CSU11,CSU13,CKP15}, we are able to refine some of them and as a consequence we obtained more general results when compared to \cite[Theorem 3.16]{CD19}. Our result is the following

\begin{thma}\label{a}	Let $\Gamma$ be an icc group satisfying one of the following conditions:
	 \begin{enumerate}
	 \item [1)] $\Gamma$ admits an unbounded, non-proper  $1$-cocycle into a mixing representation; 
	\item [2)] $\Gamma$ is exact, torsion free  and admits an unbounded quasi-cocycle into a weakly-$\ell^2$ mixing representation; 
	  
	\end{enumerate}
	For any von Neumann subalgebra  $N \subseteq L(\Gamma)$ satisfying $\Gamma \subseteq \mathcal N_{L(\Gamma)}(N) $, one can find a normal subgroup $\Sigma\lhd \Gamma$ such that $N=L(\Sigma)$.

\end{thma}

 \noindent Appealing to \cite[Lemma 3., Theorem 3.4]{T09} and \cite[Theorem 2.6]{PT11}, item 1) of Theorem \ref{a} further implies that for every icc, non-amenable groups $\Gamma$ which has positive first $L^2$-Betti number and contains an infinite amenable subgroup, $L(\Gamma)$ satisfies the ISR property; this generalizes \cite[Theorem 1.3]{AJ22} by completely removing the assumption that $\Gamma$ needs to satisfy the Atiyah conjecture and replacing the torsion free assumption with a  weaker condition. In connection to this we believe that the ISR conditions folds for all icc non-amenable groups with positive first $L^2$-Betti number but we do not have a proof in this generality. To pursue a new technique in this direction we believe it would be instrumental to tackle the torsion group constructed in \cite{Osi09}. 
\vskip 0.07in

\noindent We also point out that the conclusion of Theorem \ref{a} still holds true for all  \emph{nonamenable}, $\Gamma$-invariant subalgebras $N$ under slightly different and in some regards more general assumptions: a) any exact group $\Gamma$ that admits an unbounded quasi-cocycle into a weakly-$\ell^2$ mixing representation; b) any exact group $\Gamma$ that admits a proper array into a weakly-$\ell^2$ representation. For details we refer the reader to Theorem \ref{thm:nononameninv}.  \vskip 0.07in
\noindent We believe that the ISR property of $L(\G)$ is in fact more intimately connected with a specific algebraic structure of $\G$ that is an implicit manifestation of the aforementioned negative curvature condition. Namely we believe that $L(\G)$ satisfies the ISR property whenever $\G$ has trivial amenable radical. More precisely we conjecture the following 

\begin{conj} Let $\G$ be an icc group. If   $A\subset L(\G)$ is a diffuse abelian von Neumann subalgebra such that $\G\subset \mathcal N_{L(\G)}(A)$ then one can find an amenable normal subgroup 
$\Sigma < \G$ such that $A \subseteq L(\Sigma)$.  
     
\end{conj}

\section{Preliminaries}

\subsection{Notations and terminology}

Throughout the paper, we write $K\Subset I$  to mean that $K$ is a finite subset of $I$. Given a set $I$ we will denote by $|I|$ its cardinality. If $G$ is a group and $K, L
\subset G$ are subsets we will denote by $KL=\{gh \,|\, g\in K,h\in L\}$ and
  \[\langle K,L\rangle^{2m}=KLKL\cdots KL,\]
  where there are $m$ copies of $KL$ on the right-hand side.

\subsection{Arrays and quasicocycles on groups}
The notion of arrays was introduced by the first author and T. Sinclair in \cite{CS11}, and further studied in \cite{CSU11, CSU13, CKP15} as a conceptual tool for understanding strong solidity of $\rm II_1$ factors. Indeed, \cite[Theorem A]{CS11} shows that the group von Neumann algebra of an icc, hyperbolic group is strongly solid. The notion of arrays generalizes that of a cocycle associated with an orthogonal group representation. In fact, any quasicocycle (see below for definition) provides examples of arrays. In this section, we recall briefly the notion of arrays. In the next section we will briefly recall the notion of Gaussian dilation of $L(\G)$. Our main interest in the study of arrays and Gaussian dilation lies in understanding the structure of $\G$-invariant subalgebras of $L(\G)$. In \cite{CD19} Gaussian deformations were used to understand the structure of invariant subfactors of the group von Neumann algebras associated with ``negatively curved groups".

\vskip 0.07in
\noindent Throughout this section $\G$ will denote a countable discrete group, and $\pi: \G \rightarrow \mathcal O(\mathcal H)$ will denote an orthogonal representation of $\G$ into some Hilbert space $\mathcal H$. 
The representation $\pi$ is called \textit{mixing} if $ \langle \pi_g(\xi), \eta \rangle \rightarrow 0$ as $g \rightarrow \infty$, for all $\xi, \eta \in \mathcal H$. The representation $\pi$ is called \textit{weakly}-$\ell^2$ if it is weakly contained in the left regular representation $\ell^2_{\mathbb R}\G$. 
\vskip 0.07in

\noindent Recall that a map $q: \G \rightarrow \mathcal H$ is called a \textit{quasicocycle} for $\pi$, if there exists some constant $D>0$ such that $\|q(gh)- \pi_g(q(h))-q(g)\| \leq D$ for all $g,h \in \G$. The defect of a quasicocycle, denoted by $D(q)$, is the infimum over all such $D$. A quasicocycle with defect $0$ is just a $1$-cocycle with coefficients in $\pi$. If a quasicocycle satisfies $q(g)=-\pi_g(q(g^{-1}))$ for all $g \in \G$ then we say that the quasicocycle is \textit{anti-symmetric}. As noted by Thom in \cite[Section 5]{T09}, for any quasicocycle $q$, the quasicocycle defined by $\tilde q(g)= \frac{1}{2}(q(g)- \pi_{g^{-1}}q(g))$ is anti-symmetric, and a bounded distance away from $q$. Hence, without loss of generality, we shall assume henceforth that any quasicocycle under discussion is anti-symmetric. We denote the space of all unbounded anti-symmetric quasicocycles associated with the representation $\pi$ by $\mathcal QH^{1}_{as}(\G, \pi)$. 
\vskip 0.05in
\noindent Many remarkable groups studied in geometric or representation group theory admit quasicocycles in their left regular representations. For example whenever $\Gamma$ is an acylindrically hyperbolic group \cite{Osi16} we have that $\mathcal QH^{1}_{as}(\G, \ell^2\G)\neq 0$, \cite{HO13}. These cover all non-elementary (relatively) hyperbolic groups, most mapping class groups of finite genus surfaces, all outer automorphism groups of free groups with at least two generators, all graph product groups that do not admit direct product decompositions, etc.

\vskip 0.07in

\noindent A map $q: \G \rightarrow \mathcal H$ is called an \textit{array} into $\mathcal H$ if it satisfies the following bounded equivariance condition:
$$\underset{h \in \G}\sup \|q(gh k)-\pi_g(q(h))\| = C(g,k)< \infty \text{ for all } g,k \in \G.$$
Clearly, quasicocycles are examples of arrays. 
\vskip 0.07in
\noindent Finally, an array (or quaiscocycle) $q$ is called \emph{proper} in for every $C>0$ the ball $\mathfrak B_C=\{ g\in \G \,:\,\|q(g)\|\leq C\}$ is finite. 
\subsection{Gaussian von Neumann algebras and their deformations.} \label{sec:prelgauss}
As in the previous section, $\G$ denotes a countable, discrete group, and $\pi: \G \rightarrow \mathcal O(\mathcal H)$  denotes an orthogonal representation. Following the treatment in \cite{PS12}, the orthogonal representation $\pi$, via the Gaussian construction, gives rise to a nonatomic standard probability measure space $(X, \mu)$ such that $L^{\infty}(X, \mu)$ is generated by a family of unitaries $\{\omega(\xi): \xi \in \mathcal H \}$ (Gaussian random variables) satisfying the following relations:\\
\textbf{a)} $\omega(0)=1$, $\omega(\xi_1 + \xi_2)=\omega(\xi_1)\omega(\xi_2)$, $\omega(\xi)^*= \omega(-\xi)$ for all $\xi, \xi_1, \xi_2 \in \mathcal H$.\\
\textbf{b)} $\tau(\omega(\xi))= e^{-\|\xi \|^2}$, where $\tau$ denotes the trace on $L^{\infty}(X, \mu)$ obtained by integration.\\
Furthermore, there exists a probability measure preserving action of $\G \overset{\hat{\pi}}\ca (X, \mu)$ such that the induced action $\G \overset{\hat{\pi}}\ca L^{\infty}(X, \mu)$ also satisfies $\hat{\pi}_g(\omega(\xi))= \omega(\pi_g(\xi))$ for all $g \in \G$, and $\xi \in \mathcal H$.
The action $\G \overset{\hat{\pi}}\ca (X, \mu)$ is called the \textit{Gaussian action} associated with $\pi$. \\
Denote by $M=L(\G)$ the group von Neumann algebra of $\G$. The \textit{Gaussian dilation} of $M$ is defined as the crossed product von Neumann algebra $\tilde M= L^{\infty}(X, \mu) \rtimes \G$. We also denote by $e_M$ the orthogonal projection from $L^2(\tilde M)$ onto $L^2(M)$. \\
Let $q: \G \rightarrow \mathcal H$ be an array associated with $\pi$. As in \cite{Si10,CS11,CSU11, CKP15}, we construct a deformation arising from $q$ via exponentiation as follows. For each $t \in \mathbb R$, let $V_t \in \mathcal U(L^2(X,\mu) \otimes \ell^2(\G))$ be defined by 
$$V_t(\xi \otimes \delta_g)= \omega(tq(g))\xi \otimes \delta_g \text{ for all } \xi \in L^2(X,\mu), \text{ and } g \in \G.$$
This procedure is referred to as the \textit{Gaussian deformation} associated with $q$.\\
For future use, we record some properties of the Gaussian deformation. The reader may consult \cite{CS11, CKP15} for the proofs.
\begin{thm}\label{propgaussdef}
	Let $\G$, $q$ and $\pi$ be as above, and let $V_t$ denote the Gaussian deformations constructed above. Then the following holds:\begin{enumerate}
		\item[a)](Transversality) $V_t$ is a strongly continuous one-parameter group of unitaries satisfying the following transversality property: For each $t \in \mathbb R$, and each $\eta \in L^2(M)$, we have 
		$$ \| e_M^{\perp} \circ V_t(\eta) \|_2^2 \leq \| \eta- V_t(\eta)\|_2^2 \leq 2 \|e_M^{\perp} \circ V_t (\eta) \|_2^2. $$
		
		\item[b)](Asymptotic bimodularity) For each $x,y \in C^*_r(\G)$ we have 
		$$\lim\limits_{t \rightarrow 0} \bigg ( \underset{\eta \in L^2(M)_1} \sup \|x V_t(\eta)y -V_t(x \eta  y)\|_2 \bigg) =0.$$
		
		\item[c)] (\cite[Proposition 6.4]{CKP15})  Let $F \subset M$ be a finite subset. Denote by $\text{cu}(F)= \{\sum_i \mu_i x_i: x_i \in F, \mu_i \in \mathbb C, |\mu_i| \leq 1 \}$. Let $X \subset (M)_1$ be a set such that $e_M^{\perp} \circ V_t \rightarrow 0$ uniformly on $X$. Then $e_M^{\perp} \circ V_t \rightarrow 0$ uniformly on $\text{cu}(F) \cdot X \cdot \text{cu}(F).$
	\item[d)] (Spectral gap argument \cite{Po06,CS11}) if $N\subseteq L(\G)$ has no amenable direct, $\pi$ is weakly -$\ell^2$, and $\G$ is exact  then $e_M^\perp\circ V_t\ra 0$ uniformly on $(N'\cap L(\G))_1$ .
	\item[e)] (\cite{CS11}) For $C>0$ denote by $\mathfrak{B}_C=\{g \in \G: \|q(g)\|\leq C \}$. Let $P_{\mathfrak{B}_C}$ denote the orthogonal projection onto the Hilbert subspace spanned by $\mathfrak{B}_C$ inside $\ell^2(\G)$. Let $A \subseteq L(\G)$ be a von Neumann subalgebra, such that $V_t \rightarrow Id$ uniformly on $(A)_1$. Then for any $\varepsilon>0$ we can find $C>0$ such that $\|a-P_{\mathfrak{B}_C}(a)\|_2 < \varepsilon$ for all $a \in (A)_1$.
	\end{enumerate}
\end{thm}

\noindent For further use we also record the following technical result.

\begin{thm}\label{atomiccorner} Let $A\subseteq L(\Gamma )$ be a von Neumann subalgebra normalized by $\Gamma$ and assume that $V_t \rightarrow Id$ uniformly on $(A)_1$. Then $A$ is not diffuse.  
    
\end{thm}

\begin{proof} Assume by contradiction that $A$ is diffuse. 

\noindent Also, let $D$ be the defect of the quasicocycle $q$. Fix $\varepsilon >0$. Since $V_t \rightarrow Id$ uniformly on $(A)_1$ from part e) in the prior result one can find $C>0$ such that \begin{equation}\label{eq:closeprojectionofelements'}\| a-P_{\mathfrak B_c}(a)\|_2<\varepsilon\text{ for all }a\in (A)_1.\end{equation}
 As $q$ is unbounded, there exists $g \in \G \setminus \mathfrak{B}_{2C + 2D}$. By \cite[Theorem 3.1]{CSU13}, there exists $K \subset \mathfrak{B}_C$, a finite subset, such that 
\begin{equation} \label{eq:csu13thm3.1}
g(\mathfrak{B}_C \setminus K) \cap (\mathfrak{B}_C \setminus K)g = \emptyset.
\end{equation}

\noindent As $A$ is diffuse, there is $a \in \mathcal U(A)$ such that $\|\mathcal P_{K}(a)\|_2 < \varepsilon$, and $\|\mathcal P_{K}(u_gau_g^*)\|_2 < \varepsilon$. Combining these with inequality~\eqref{eq:closeprojectionofelements'} we get

\begin{equation}\label{eq:closeprojunitaryconj2}\begin{split}\|a- \mathcal P_{\mathfrak{B}_C \setminus K}(a) \|_2 &= \|a- \mathcal P_{\mathfrak{B}_C} \circ \mathcal P_{\G \setminus K}(a) \|_2 
=\|a- \mathcal P_{\mathfrak{B}_C} (a- \mathcal P_K(a)) \|_2 \\ & \leq \|a- \mathcal P_{\mathfrak{B}_C} (a) \|_2 + \|\mathcal P_{\mathfrak{B}_C} \circ \mathcal P_K (a) \leq \varepsilon + \| \mathcal P_K(a)\|_2 \leq 2 \varepsilon.\end{split}\end{equation}

\noindent Similarly, we have
\begin{equation} \label{eq:closeprojunitaryconj}
\| u_gau_g^* - \mathcal P_{\mathfrak{B}_C \setminus K}(u_g a u_g^*) \| \leq 2 \varepsilon.
\end{equation}
Then basic estimates together with inequalities~\eqref{eq:closeprojunitaryconj}-\eqref{eq:closeprojunitaryconj2}  and equation~\eqref{eq:csu13thm3.1} show that
\begin{align*}
1&= \|u_g\|_2^2 = |\langle u_g , u_g \rangle | = |\langle u_g au_g^* u_g, u_g a \rangle|\\&
 \leq  |\langle \mathcal P_{\mathfrak{B}_C \setminus K}(u_gau_g^*)u_g, u_g a \rangle| + \|\mathcal P_{\mathfrak{B}_C \setminus K}(u_gau_g^*)-u_gau_g^* \|_2 \\
 & \leq |\langle \mathcal P_{\mathfrak{B}_C \setminus K}(u_gau_g^*)u_g, u_g \mathcal P_{\mathfrak{B}_C \setminus K}(a) \rangle| + \|\mathcal P_{\mathfrak{B}_C \setminus K}(u_gau_g^*)-u_gau_g^* \|_2 + \\&\quad + \|\mathcal P_{\mathfrak{B}_C \setminus K}(u_gau_g^*) \|_2 \| \mathcal P_{\mathfrak{B}_C \setminus K}(a)-a\|_2 \\&\leq 0 + 4 \varepsilon,
\end{align*}
which is a contradiction for $\varepsilon < \frac{1}{4}$. \end{proof}

\subsection{Acylindrically hyperbolic groups}

\noindent The notion of an acylindrically hyperbolic group, introduced by Osin \cite{Osi16} as a generalization of non-elementary hyperbolic and relatively hyperbolic groups, is defined using the notion of an acylindrical action, which was introduced in \cite{Sel92} for actions on trees, and in \cite{Bow08} for actions on general metric spaces.

\begin{defn}
    Let $\Gamma$ be a group acting on a metric space $(S,d)$ by isometries. We say the action is \textit{acylindrical} if for every $D>0$, there exists $R,N>0$ such that if two points $x,y\in S$ satisfy $d(x,y)>R$, then
    \[|\{g\in G\mid d(x,gx)<D \text{ and } d(y,gy)<D\}|<N.\]
\end{defn}

\begin{defn}
    A group $\Gamma$ is \textit{acylindrically hyperbolic} if it admits a non-elementary acylindrical action on a Gromov hyperbolic space $S$.
\end{defn}
\vskip 0.07in  
\noindent In the above, $S$ is \textit{Gromov hyperbolic} if it is a geodesic metric space (i.e., every two points can be connected by a geodesic) and there exists a constant $\delta\geqslant 0$ such that whenever we have a triangle in $S$ with geodesic sides $a,b,c$, the side $c$ is contained in the $\delta$-neighborhood of $a\cup b$ and similarly for $a$ and $b$; the action of $\Gamma$ on $S$ is \textit{non-elementary} if the limit set of $\Gamma$ on the Gromov boundary of $S$ has at least $3$ points, which is equivalent to that $\Gamma$ has unbounded orbits and is not virtually cyclic \cite[Theorem 1.1]{Osi16}.
\vskip 0.05in 
\noindent Acylindrically hyperbolic groups can be equivalently defined using the notion of a hyperbolically embedded subgroup, which we recall below.

\vskip 0.07in 
\noindent Throughout this paper, we think of graphs as metric spaces. Given a graph $S$, we think of every edge of $S$ to have length $1$ and the distance between two points $x,y\in S$ is the length of a shortest path between $x,y$.

\vskip 0.07in 
\noindent Now, let $\Gamma$ be a group with a subgroup $C$. Fix a subset $X\subseteq \Gamma$ such that $X\cup C$ generates $\Gamma$ as a group. Consider the Cayley graph $\mathrm{Cay}(\Gamma,X\sqcup C)$. Here, for technical purposes we use disjoint union instead of union. We emphasize that we do allow non-trivial intersection between $X$ and $C$, and when this happens we will have multiple edges between certain pairs of vertices in $\mathrm{Cay}(\Gamma,X\sqcup C)$. Note also that the Cayley graph $\mathrm{Cay}(C,C)$ can be identified with a subgraph of $\mathrm{Cay}(\Gamma,X\sqcup C)$, i.e., the subgraph whose vertices and edges are all labeled by $C$.
\vskip 0.07in 
\noindent An edge path $p\subseteq \mathrm{Cay}(\Gamma,X\sqcup C)$ is called \textit{$C$-admissible} if $p$ does not contain edges of $\mathrm{Cay}(C,C)$. Note that we allow a $C$-admissible path $p$ to contain vertices of $\mathrm{Cay}(C,C)$. Using admissible paths we can define a \textit{relative metric} $\widehat d_{C}$ on $C$: for every pair of elements $c_1,c_2\in C$, $\widehat d_{C}(c_1,c_2)$ equals the length of a shortest $C$-admissible path between the vertices of $\mathrm{Cay}(\Gamma,X\sqcup C)$ labeled by $c_1,c_2$, if such a path exists; or $\infty$ otherwise. The arithmetic laws of $[0,\infty)$ extend naturally to $[0,\infty]$ and it is easy to verify that $\widehat d_{C}\colon C\times C\rightarrow [0,\infty]$ defines a metric on $C$. If $p$ is an edge path of $\mathrm{Cay}(\Gamma,X\sqcup C)$ whose edges are labeled by elements of $C$, then we let $\widehat{\ell}_{C}(p)=\widehat d_{C}(1,c^{-1}_1c_2)$, where $c_1$ (resp. $c_2$) is the labeled of the initial (resp. terminal) vertex of $p$.

\begin{defn}\label{def. hyperbolically embed}
	We say that a subgroup $C$ \textit{hyperbolically embeds} into a group $\Gamma$ with respect to a set $X\subseteq \Gamma$, denoted as $C \hookrightarrow_h (\Gamma,X)$, if the following hold.
	\begin{enumerate}[label=(\roman*)]
		\item\label{item. he0} $X\cup C$ generates $\Gamma$ as a group.
		\item\label{item. he1} $\mathrm{Cay}(\Gamma,X\sqcup C)$ is Gromov hyperbolic.
		\item\label{item. he2} The metric $\widehat d_C$ is locally finite, i.e., every ball of finite radius contains only finitely many elements.
	\end{enumerate}
	We say that $C$ \textit{hyperbolically embeds} into $\Gamma$ if there exists a set $X\subseteq \Gamma$ such that \ref{item. he0}, \ref{item. he1} and \ref{item. he2} hold.
\end{defn}

The following is the main result of \cite{Osi16}.
\begin{thm}[Osin, 2016]
A group $\Gamma$ is acylindrically hyperbolic if and only if $\Gamma$ has a proper infinite hyperbolically embedded subgroup.
\end{thm}

\noindent Now suppose we have groups $\Gamma\geqslant C$ and a subset $X\subseteq \Gamma$ such that $C\hookrightarrow_h (\Gamma,X)$. Let $p$ be a path in $\mathrm{Cay}(\Gamma,X\sqcup C)$. A \textit{$C$-subpath} $q$ of $p$ is a subpath of $p$ all of whose edges are labeled by elements of $C$ (if $p$ is a cycle then we  allow $q$ to be a subpath of a cyclic shift of $p$). Also, $q$ is called a \textit{$C$-component} if $q$ is a $C$-subpath and $q$ is not properly contained in any other $C$-subpath. Two $C$-components $q_1,q_2$ of $p$ are \textit{connected} if there is an edge path $q_3$ all of whose edges are labeled by elements of $C$ such that $q_3$ connects a vertex of $q_1$ to a vertex $q_2$. If a $C$-component $q$ is not connected to any other $C$-components, then we say $q$ is an \textit{isolated $C$-component}. A useful property of isolated $C$-components is that, in a geodesic polygon, the total $\widehat{\ell}_{C}$-length of the isolated $C$-components is bounded uniformly in the number of sides of the polygon. The following is a simplified version of \cite[Proposition 4.14]{DGO17}, which generalizes \cite[Proposition 3.2]{Osi07}.

\begin{prop}[Dahmani--Guirardel--Osin, 2017]\label{prop. isolated component}
	There exists $D>0$ such that the following holds. Let $p$ be an $n$-gon in $\mathrm{Cay}(\Gamma,X\sqcup C)$ with geodesic sides and let $I$ be the subset of sides consisting of isolated $C$-components of $p$. Then \[\sum_{q\in I}\widehat{\ell}_{C}(q)\leq Dn.\]
\end{prop}

\begin{thm}\label{8gon}
	Let $\Gamma$ be a group with a hyperbolically embedded subgroup $C$. Then for every $K\Subset \Gamma\setminus C$, there exists $L\Subset C$ such that for all $m,n\in\mathbb N^+$,
	\begin{equation}\label{eq. polygon}
	\langle (C\setminus L),K \rangle^{2m} \cap \langle K,(C\setminus L) \rangle^{2n}=\emptyset.
	\end{equation}
\end{thm}
\begin{proof}
	Fix $X\subset \Gamma$ such that $C\hookrightarrow_h(\Gamma,X)$. For each $k\in K$, there exists a geodesic word $w_k$ over $(X\cup X^{-1})\sqcup C$ representing $k$ such that the following holds.
	\begin{itemize}
		\item Let $u_k$ (resp. $v_k$) be the maximal initial (resp. terminal) segment of $w_k$ labeled by a word over $C$. Write $w_k$ as the concatenation of $u_k,w'_k$ and $v_k$. Then no non-trivial initial or terminal segment of $w'_k$ can represent an element of $C$.
	\end{itemize}
	
	\noindent We note that $w'_k\neq \emptyset$ for all $k$ as $K\cap C=\emptyset$. Let $R=\{u_k,u^{-1}_k,v_k,v^{-1}_k\mid k\in K\}$, let $D$ be the constant given by Proposition \ref{prop. isolated component} and let $L\Subset C$ be the finite set consisting of elements $c\in C$ such that there exists $r_1,r_2\in R$ with
	\begin{equation}\label{eq. long component}
	\widehat d_C(1,r_1cr_2)\leq 4D.
	\end{equation}
	Note that $L=L^{-1}$.
	
	\begin{claim}\label{noproduct}
		There do not exist $c_1,\cdots, c_{m+n}\in C\setminus L$ and $k_1,\cdots,k_{m+n}\in K\cup K^{-1}$ such that
		\[\prod^{m+n}_{i=1} c_ik_i=1.\]
	\end{claim}
	
	\noindent \emph{Proof of Claim \ref{noproduct}.} Suppose for the contrary that such a set of $c_i,k_i$ exists. For simplicity, we will assume that $k_i\in K$ for all $1\leq i \leq m+n$. The general case can be treated in the same way.
	
	\noindent For $1\leqslant i\leqslant m+n$, let $c'_i\in C$ be the element represented by the word $v_{k_{i-1}}c_iu_{k_i}$ (subscript modulo $m+n$). Consider the path $p\subseteq \mathrm{Cay}(\Gamma,X\sqcup C)$ labeled by the word $\prod^{m+n}_{i=1} c'_iw'_{k_i}$. For $1\leqslant i\leqslant m+n$, let $p_i$ (resp. $q_i$) be the subpath of $p$ labeled by the word $c'_i$ (resp. $w'_{k_i}$). Then $p$ is a $(2m+2n)$-gon with geodesic sides $p_i,q_i$; here we used that $w'_{k_i}\neq\emptyset$.
	
	\begin{figure}
		\centering
		
		\begin{minipage}{0.45\textwidth}
			\centering
			
			\tikzset{every picture/.style={line width=0.75pt}} 
			
			\begin{tikzpicture}[x=0.75pt,y=0.75pt,yscale=-0.5,xscale=0.5]
			
			\draw [color={rgb, 255:red, 0; green, 0; blue, 0 }  ,draw opacity=0.3 ][line width=3]    (120.5,93.53) -- (218.91,36.45) ;
			\draw [shift={(223.23,33.95)}, rotate = 149.89] [color={rgb, 255:red, 0; green, 0; blue, 0 }  ,draw opacity=0.3 ][line width=3]    (20.77,-6.25) .. controls (13.2,-2.65) and (6.28,-0.57) .. (0,0) .. controls (6.28,0.57) and (13.2,2.66) .. (20.77,6.25)   ;
			\draw [line width=3]    (223.23,33.95) -- (398.92,33.03) ;
			\draw [shift={(403.92,33)}, rotate = 179.7] [color={rgb, 255:red, 0; green, 0; blue, 0 }  ][line width=3]    (20.77,-6.25) .. controls (13.2,-2.65) and (6.28,-0.57) .. (0,0) .. controls (6.28,0.57) and (13.2,2.66) .. (20.77,6.25)   ;
			\draw [line width=3]    (520.5,98.26) -- (520.5,187.85) ;
			\draw [shift={(520.5,192.85)}, rotate = 270] [color={rgb, 255:red, 0; green, 0; blue, 0 }  ][line width=3]    (20.77,-6.25) .. controls (13.2,-2.65) and (6.28,-0.57) .. (0,0) .. controls (6.28,0.57) and (13.2,2.66) .. (20.77,6.25)   ;
			\draw [line width=3]    (403.92,259.05) -- (250.09,259.97) ;
			\draw [shift={(245.09,260)}, rotate = 359.66] [color={rgb, 255:red, 0; green, 0; blue, 0 }  ][line width=3]    (20.77,-6.25) .. controls (13.2,-2.65) and (6.28,-0.57) .. (0,0) .. controls (6.28,0.57) and (13.2,2.66) .. (20.77,6.25)   ;
			\draw [color={rgb, 255:red, 0; green, 0; blue, 0 }  ,draw opacity=0.3 ][line width=3]    (520.5,192.85) -- (408.27,256.58) ;
			\draw [shift={(403.92,259.05)}, rotate = 330.41] [color={rgb, 255:red, 0; green, 0; blue, 0 }  ,draw opacity=0.3 ][line width=3]    (20.77,-6.25) .. controls (13.2,-2.65) and (6.28,-0.57) .. (0,0) .. controls (6.28,0.57) and (13.2,2.66) .. (20.77,6.25)   ;
			\draw [color={rgb, 255:red, 0; green, 0; blue, 0 }  ,draw opacity=0.3 ][line width=3]    (403.92,33) -- (516.14,95.82) ;
			\draw [shift={(520.5,98.26)}, rotate = 209.24] [color={rgb, 255:red, 0; green, 0; blue, 0 }  ,draw opacity=0.3 ][line width=3]    (20.77,-6.25) .. controls (13.2,-2.65) and (6.28,-0.57) .. (0,0) .. controls (6.28,0.57) and (13.2,2.66) .. (20.77,6.25)   ;
			\draw [line width=3]    (122.69,194.74) -- (120.61,98.53) ;
			\draw [shift={(120.5,93.53)}, rotate = 88.76] [color={rgb, 255:red, 0; green, 0; blue, 0 }  ][line width=3]    (20.77,-6.25) .. controls (13.2,-2.65) and (6.28,-0.57) .. (0,0) .. controls (6.28,0.57) and (13.2,2.66) .. (20.77,6.25)   ;
			\draw [color={rgb, 255:red, 0; green, 0; blue, 0 }  ,draw opacity=0.3 ][line width=3]    (245.09,260) -- (127.1,197.09) ;
			\draw [shift={(122.69,194.74)}, rotate = 28.07] [color={rgb, 255:red, 0; green, 0; blue, 0 }  ,draw opacity=0.3 ][line width=3]    (20.77,-6.25) .. controls (13.2,-2.65) and (6.28,-0.57) .. (0,0) .. controls (6.28,0.57) and (13.2,2.66) .. (20.77,6.25)   ;
			\draw [line width=3]  [dash pattern={on 7.88pt off 4.5pt}]  (120.5,93.53) .. controls (144.77,232.59) and (324.5,152) .. (223.23,33.95) ;
			
			\draw (85,137) node [anchor=north west][inner sep=0.75pt]    {$p_{1}$};
			\draw (145,37) node [anchor=north west][inner sep=0.75pt]    {$q_{1}$};
			\draw (296,4) node [anchor=north west][inner sep=0.75pt]    {$p_{2}$};
			\draw (458,35) node [anchor=north west][inner sep=0.75pt]    {$q_{2}$};
			\draw (531,129) node [anchor=north west][inner sep=0.75pt]    {$p_{3}$};
			\draw (473,224) node [anchor=north west][inner sep=0.75pt]    {$q_{3}$};
			\draw (324,265) node [anchor=north west][inner sep=0.75pt]    {$p_{4}$};
			\draw (166,230) node [anchor=north west][inner sep=0.75pt]    {$q_{4}$};
			\draw (201,130) node [anchor=north west][inner sep=0.75pt]    {$t$};
			
			\end{tikzpicture}
			
		\end{minipage}
		\begin{minipage}{0.45\textwidth}
			\centering

			\tikzset{every picture/.style={line width=0.75pt}} 
			
			\begin{tikzpicture}[x=0.75pt,y=0.75pt,yscale=-0.5,xscale=0.5]
			
			\draw [color={rgb, 255:red, 0; green, 0; blue, 0 }  ,draw opacity=0.3 ][line width=3]    (120.5,93.53) -- (218.91,36.45) ;
			\draw [shift={(223.23,33.95)}, rotate = 149.89] [color={rgb, 255:red, 0; green, 0; blue, 0 }  ,draw opacity=0.3 ][line width=3]    (20.77,-6.25) .. controls (13.2,-2.65) and (6.28,-0.57) .. (0,0) .. controls (6.28,0.57) and (13.2,2.66) .. (20.77,6.25)   ;
			\draw [line width=3]    (223.23,33.95) -- (398.92,33.03) ;
			\draw [shift={(403.92,33)}, rotate = 179.7] [color={rgb, 255:red, 0; green, 0; blue, 0 }  ][line width=3]    (20.77,-6.25) .. controls (13.2,-2.65) and (6.28,-0.57) .. (0,0) .. controls (6.28,0.57) and (13.2,2.66) .. (20.77,6.25)   ;
			\draw [line width=3]    (520.5,98.26) -- (520.5,187.85) ;
			\draw [shift={(520.5,192.85)}, rotate = 270] [color={rgb, 255:red, 0; green, 0; blue, 0 }  ][line width=3]    (20.77,-6.25) .. controls (13.2,-2.65) and (6.28,-0.57) .. (0,0) .. controls (6.28,0.57) and (13.2,2.66) .. (20.77,6.25)   ;
			\draw [line width=3]    (403.92,259.05) -- (250.09,259.97) ;
			\draw [shift={(245.09,260)}, rotate = 359.66] [color={rgb, 255:red, 0; green, 0; blue, 0 }  ][line width=3]    (20.77,-6.25) .. controls (13.2,-2.65) and (6.28,-0.57) .. (0,0) .. controls (6.28,0.57) and (13.2,2.66) .. (20.77,6.25)   ;
			\draw [color={rgb, 255:red, 0; green, 0; blue, 0 }  ,draw opacity=0.3 ][line width=3]    (520.5,192.85) -- (408.27,256.58) ;
			\draw [shift={(403.92,259.05)}, rotate = 330.41] [color={rgb, 255:red, 0; green, 0; blue, 0 }  ,draw opacity=0.3 ][line width=3]    (20.77,-6.25) .. controls (13.2,-2.65) and (6.28,-0.57) .. (0,0) .. controls (6.28,0.57) and (13.2,2.66) .. (20.77,6.25)   ;
			\draw [color={rgb, 255:red, 0; green, 0; blue, 0 }  ,draw opacity=0.3 ][line width=3]    (403.92,33) -- (516.14,95.82) ;
			\draw [shift={(520.5,98.26)}, rotate = 209.24] [color={rgb, 255:red, 0; green, 0; blue, 0 }  ,draw opacity=0.3 ][line width=3]    (20.77,-6.25) .. controls (13.2,-2.65) and (6.28,-0.57) .. (0,0) .. controls (6.28,0.57) and (13.2,2.66) .. (20.77,6.25)   ;
			\draw [line width=3]    (122.69,194.74) -- (120.61,98.53) ;
			\draw [shift={(120.5,93.53)}, rotate = 88.76] [color={rgb, 255:red, 0; green, 0; blue, 0 }  ][line width=3]    (20.77,-6.25) .. controls (13.2,-2.65) and (6.28,-0.57) .. (0,0) .. controls (6.28,0.57) and (13.2,2.66) .. (20.77,6.25)   ;
			\draw [color={rgb, 255:red, 0; green, 0; blue, 0 }  ,draw opacity=0.3 ][line width=3]    (245.09,260) -- (127.1,197.09) ;
			\draw [shift={(122.69,194.74)}, rotate = 28.07] [color={rgb, 255:red, 0; green, 0; blue, 0 }  ,draw opacity=0.3 ][line width=3]    (20.77,-6.25) .. controls (13.2,-2.65) and (6.28,-0.57) .. (0,0) .. controls (6.28,0.57) and (13.2,2.66) .. (20.77,6.25)   ;
			\draw [line width=3]  [dash pattern={on 7.88pt off 4.5pt}]  (120.5,93.53) .. controls (256.5,148) and (377.5,148) .. (520.5,98.26) ;
			
			\draw (85,137) node [anchor=north west][inner sep=0.75pt]    {$p_{1}$};
			\draw (145,37) node [anchor=north west][inner sep=0.75pt]    {$q_{1}$};
			\draw (296,4) node [anchor=north west][inner sep=0.75pt]    {$p_{2}$};
			\draw (458,35) node [anchor=north west][inner sep=0.75pt]    {$q_{2}$};
			\draw (531,129) node [anchor=north west][inner sep=0.75pt]    {$p_{3}$};
			\draw (473,224) node [anchor=north west][inner sep=0.75pt]    {$q_{3}$};
			\draw (324,265) node [anchor=north west][inner sep=0.75pt]    {$p_{4}$};
			\draw (166,230) node [anchor=north west][inner sep=0.75pt]    {$q_{4}$};
			\draw (314,98) node [anchor=north west][inner sep=0.75pt]    {$t$};

			\end{tikzpicture}
			
		\end{minipage}
		
		\caption{The case $m=n=2$: the geodesic $8$-gon $p$ with two possibilities of the path $t$}
		\label{fig. 1}
		
	\end{figure}
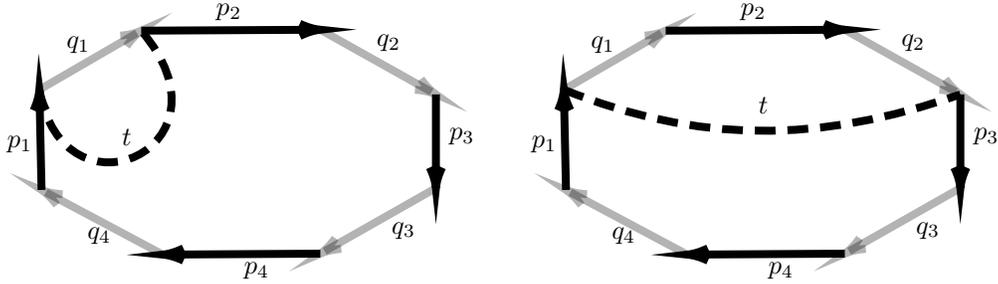
	\vskip 0.07in 
	\noindent For all $i$, the sides $p_i$ are $C$-components of $p$, and inequality \eqref{eq. long component} implies that $\widehat{\ell}_C(p_i)>4D$. So Proposition \ref{prop. isolated component} implies that the $C$-components $p_i$ cannot all be isolated in $p$. Up to relabeling the subpaths $p_i,q_i$, we may assume:
	\begin{itemize}
		\item $p_1$ is connected to either $p_k$ or some $C$-component of $q_{k-1}$ for some $k$. In the former case, let $s=p_k$ and in the latter case, let $s$ be the corresponding $C$-component of $q_{k-1}$. Let $t$ be an edge labeled by an element of $C$ connecting the terminal vertex of $p_1$ to the initial vertex of $s$, let $q'_{k-1}$ be the segment of $q_{k-1}$ from the initial vertex of $q_{k-1}$ to the initial vertex of $s$, and let $q$ be the polygon formed by $q_1,p_2,q_2,\cdots,p_{k-1},q'_{k-1}$ and $t$. Then $p_2,\cdots, p_{k-1}$ are isolated $C$-components of $q$ (if some $p_i,i\in\{2,3,\cdots,k-1\}$, is not isolated, simply rename $p_i$ to be $p_1$).
	\end{itemize}
	We refer to Figure \ref{fig. 1} for an illustration of the case $m=n=2$.
\vskip 0.07in 	
	\noindent If $k=2$, then some non-trivial initial segment of $w'_1$ represents an element of $C$, a contradiction. If $k>2$, then $q$ is a $(2k-2)$-gon with geodesic sides. $q$ has $k-2$ isolated $C$-components $p_2,\cdots, p_{k-1}$. Inequality \eqref{eq. long component} and Proposition \ref{prop. isolated component} then implies
	\[4D(k-2)<(2k-2)D,\]
	which contradicts the assumption $k>2$.$\hfill\blacksquare$
	\vskip 0.05in 
	\noindent Equation \eqref{eq. polygon} is an immediate consequence of the above claim.
\end{proof}

\subsection{Popa's intertwining techniques}
More than fifteen years ago, S. Popa  introduced  in \cite [Theorem 2.1 and Corollary 2.3]{Po03} a powerful analytic criterion for identifying intertwiners between arbitrary subalgebras of tracial von Neumann algebras. This is known as \emph{Popa's intertwining-by-bimodules technique} in the current literature. Popa's intertwining-by-bimodules technique has played a key role in the classification of von Neumann algebras program via Popa's deformation/rigidity theory. 

\begin{thm}\cite{Po03} \label{corner} Let $(\mathcal M,\tau)$ be a separable tracial von Neumann algebra and let $\mathcal P, \mathcal Q\subseteq \mathcal M$ be (not necessarily unital) von Neumann subalgebras. 
	Then the following are equivalent:
	\begin{enumerate}
		\item There exist $ p\in  \mathcal P(\mathcal P), q\in  \mathcal P(\mathcal Q)$, a $\ast$-homomorphism $\theta:p \mathcal P p\rightarrow q\mathcal Q q$  and a partial isometry $0\neq v\in q \mathcal M p$ such that $\theta(x)v=vx$, for all $x\in p \mathcal P p$.
		\item For any group $\mathcal G\subset \mathcal U(\mathcal P)$ such that $\mathcal G''= \mathcal P$ there is no sequence $(u_n)_n\subset \mathcal G$ satisfying $\|E_{ \mathcal Q}(xu_ny)\|_2\rightarrow 0$, for all $x,y\in \mathcal  M$.
	\end{enumerate}
\end{thm} 
\vskip 0.02in
\noindent If one of the two equivalent conditions from Theorem \ref{corner} holds then we say that \emph{ a corner of $\mathcal P$ embeds into $\mathcal Q$ inside $\mathcal M$}, and write $\mathcal P\prec_{\mathcal M}\mathcal Q$. If we moreover have that $\mathcal P p'\prec_{\mathcal M}\mathcal Q$, for any projection  $0\neq p'\in \mathcal P'\cap 1_{\mathcal P} \mathcal M 1_{\mathcal P}$ (equivalently, for any projection $0\neq p'\in\mathcal Z(\mathcal P'\cap 1_{\mathcal P}  \mathcal M 1_{P})$), then we write $\mathcal P\prec_{\mathcal M}^{s}\mathcal Q$. We refer the readers to the survey papers \cite{Po07,Vaicm,Io18} for recent progress in von Neumann algebras using deformation/rigidity theory.
\vskip 0.05in 
\noindent For future use, we recall a special case of \cite[Corollary 7]{HPV10}, that generalizes \cite[Theorem 6.16]{PV08}.
\begin{prop}[Houdayer-Popa-Vaes] \label{prop:hpvcor7}
Let $\G \curvearrowright (\mathcal A, \tau)$ be a trace preserving action, and let $M=\mathcal A \rtimes \G$. Let $B \subseteq M $ be a regular von Neumann algebra, and let $\Sg < \G$ be a subgroup such that $B \preceq_{M}  \mathcal A \rtimes \Sg$. Then, $B \preceq_{M} \mathcal A \rtimes (\cap_i g_i\Sg g_i^{-1})$ for all $g_1, \ldots, g_n \in \G$.
\end{prop}

\section{Some general results on $\Gamma$-invariant subalgebras of $L(\Gamma)$}

 In this section we collect together a few general facts concerning $\Gamma$-invariant subalgebras of $L(\Gamma)$.

\noindent The first result, is closely related to \cite[Theorem 3.15]{CD19}, and we include it here just for completeness as it will be used in the sequel.

\begin{thm} \label{thm: invariant factors}
	Let $\G$ be an icc group, and $B \subseteq L(\G)$ be a subfactor such that $\G \subseteq \mathcal N_{L(\G)}(B)$. Then, there exists a normal subgroup $\Sigma \unlhd \G$ such that $B \vee (B' \cap L(\Sigma))= L(\Sigma)$. Moreover, one can find collections of unitaries $\{v_g\}_{g\in \Sg}\subset \mathcal U(B)$ and $\{w_g\}_{g\in \Sg}\subset \mathcal U(B'\cap L(\Sigma))$ and a $2$-cocycle c: $\Sigma \times \Sg \rightarrow \mathbb T$ such that for all $g,h\in \Sg$ the following relations hold:
 \begin{eqnarray}
&& u_g=v_g w_g  \\
&& v_g v_h = c_{g,h} v_{gh}\\
&& w_gw_h = \overline{c_{g,h}} w_{gh} 
	\end{eqnarray}
\end{thm}

\begin{proof}Using \cite[Theorem 3.15]{CD19}, one can find a normal subgroup $\Sg \lhd \G$ such that $B \subseteq L(\Sg) \subseteq B \vee B' \cap L(\G)$. For the reader's convenience we recall next the argument from \cite{CD19}. Consider the set $S=\{ g \in \G: E_B(u_{g}) \neq 0\}$ and notice it satisfies $S=S^{-1}$ and $1\in S$. Fix $g \in S$ and consider the inner automorphism of $L(\G)$ given by $\alpha={\rm ad}(u_{g})$. Note that $\alpha$ restricts to a $\ast$-automorphism of $B$. Thus we have that
	$\alpha(x)E_B(u_{g})=E_B(u_{g})x$ for all $x \in B$. This implies $b_{g}:=E_B(u_{g})u_{g}^* \in B' \cap L(\G)$. 
From the choice of $g$ we have $b_{g} \neq 0$.
	Thus we have $E_B(u_{g})=b_{g}u_{g}$. Hence, $E_B(u_{g})E_{B}(u_{g})^{\ast}=b_{g}b_{g}^*$. Therefore, $b_{g}b_{g}^* \in B' \cap B = \mathbb C$, which implies $b_{g}b_{g}^* =\tau(b_{g}b_{g}^*)1 $. Hence, by normalizing if necessary, we can find a unitary $c_{g} \in \mathcal U(B)$ such that $E_B(u_{g})= \|b_{g}\|_2 c_{g}$. As $E_B(u_{g})=b_{g}u_{g}$, we get $u_{g}= \|b_{g}\|_2 b_{g}^*c_{g}$. Note that $\|b_{g}\|_2 b_{g}^* \in \euu(B ' \cap L(\G))$. So  $u_{g} \in \euu(B) \euu(B' \cap L(\G)) \subseteq B \vee B' \cap L(\G)$.
	  
\vskip 0.04in
	
	\noindent Next let $\Sg$ be the set of all $g \in \G$ such that \begin{equation}\label{eq:uvwreln}u_{g}=v_{g}w_{g}\end{equation}where $v_{g}\in \euu(B)$, and $w_{g}\in \euu(B' \cap L(\G))$. We notice that $\Sg$ is in fact a normal subgroup of $L(\G)$ containing the set $S$, and hence $B \subseteq L(\Sg)$. Since $L(\Sg) \subseteq B \vee (B' \cap L(\G))$ by construction, we get $B \subseteq L(\Sg) \subseteq B \vee B' \cap L(\G)$.	Furthermore, since $B$ is a factor, using Ge's tensor splitting theorem \cite{Ge}, we get $L(\Sg)= B \vee (B' \cap L(\Sg))$. 
	\vskip 0.05in 

\noindent Fix  $g, h \in \Sg$. Using \eqref{eq:uvwreln} we have $v_{g h} w_{g h}=u_{g h}=u_{g}u_{h}= v_{g}w_{g}v_{h}w_{h}$ which further implies $v_{g h}^* v_{g}v_{h}= w_{g h} w_{h}^* w_{g}^*$. Notice the left-hand side of this equation belongs to $B$ while the right-hand side belongs to $B' \cap L(\Sg)$. Since $B$ is a factor, we necessarily have that 
	\begin{equation} \label{eq:cocycle}
	v_{g h}^* v_{g}v_{h}= w_{g h} w_{h}^* w_{g}^* = c_{g,h} \in \mc.
	\end{equation}
Moreover, this equation also entails that $c_{g, h}= \tau(v_{g h}^* v_{g}v_{h})$ and   $|c_{g, h}|=\|v_{g h}^* v_{g}v_{h}\|_2=1.$
\vskip 0.05in 

\noindent In conclusion, the $\mathbb T$-valued $2$-cocycle relations stated in the statement are satisfied. 
\end{proof}

\noindent For further use we also record the following result. 

\begin{prop} \label{nofindiminv}
Let $M=L(\G)$ be a $\rm II_1$ factor. Then any finite dimensional subspace  $A \subseteq M$ that is invariant under the conjugation action by $\G$ satisfies  $A =\mathbb C 1$. 
\end{prop}

\begin{proof} Assume by contradiction $A \neq \mathbb C1$. Consider $B= \{ x-\tau(x)1\,:\,x\in A \}\subset M$ and note it is a nontrivial finite dimensional subspace that is invariant under  conjugating by elements of $\Gamma$. Pick $b_1,..., b_k \subset B$, where $k\geq 1$, a (finite) orthonormal basis with respect to the dot product induced by the trace of $M$. Thus for every $g\in \Gamma$ and $i=1,...,k$ one can find scalars  $\alpha(g,i)\in \mathbb C$ with $\sup_{g,i} |\alpha(b,i)|\leq 1$ so that  
\begin{equation}\label{smallorbit}
    u_g b_1u_g^{-1}=\sigma_g(b_1)=\sum^k_{i=1} \alpha(g,i) b_i.
\end{equation}

 \noindent Fix $\varepsilon>0$. Using basic $\|\cdot\|_2$-approximations there is a finite set $F_\varepsilon\subset \Gamma \setminus\{1\}$ such that for every $i=1,...,k$ there is $b^\varepsilon_i\in M$ satisfying the following properties: \begin{eqnarray}
\label{normbound}&&\|b_i -b^{\varepsilon}_i\|_2\leq \varepsilon\\
\label{supbound}&& \text{its support }\sup (b^\epsilon_i)\subseteq F_\varepsilon.
\end{eqnarray}

\noindent As $\Gamma$ is icc one can find $g\in \Gamma$ such that $gF_\varepsilon g^{-1}\cap F_\varepsilon=\emptyset$; in particular, by \eqref{supbound} we have $\langle u_gb^\varepsilon_1u_{g^{-1}}, b^\varepsilon_i\rangle=0$, for all $i=1,...,k$.
 This combined with \eqref{smallorbit}, \eqref{normbound} and Cauchy-Schwartz inequality show that 

\begin{equation}\begin{split}
    1&= \| b_1\|_2^2= \| u_gb_1u_{g^{-1}}\|_2^2= \langle u_gb_1u_{g^{-1}},  u_gb_1u_{g^{-1}}\rangle =\\
 & =\sum^k_{i=1} \alpha(g,i)\langle u_gb_1u_{g^{-1}}, b_i\rangle  \leq \sum^k_{i=1} |\langle u_gb_1u_{g^{-1}}, b_i\rangle|\\
 & \leq \varepsilon k +\sum^k_{i=1} |\langle u_gb^\varepsilon_1u_{g^{-1}}, b_i\rangle| 
 \leq \varepsilon k + \varepsilon (1+\varepsilon) k+ \sum^k_{i=1} |\langle u_gb^\varepsilon_1u_{g^{-1}}, b^\varepsilon_i\rangle|\\
 & =(2\varepsilon+\varepsilon^2)k.\end{split}
 \end{equation}
 This however leads to a contradiction when $\varepsilon< (3k)^{-1}$. 
 \end{proof}

\begin{cor} \label{thm:invariantabelian}
Let $M=L(\G)$ be a $\rm II_1$ factor. Then for a von Neumann subalgebra  $A \subseteq M$ satisfying $\G \subseteq \mathcal N_M(A)$ we have that $A$ is either trivial, or diffuse. 	
\end{cor}

\begin{proof}
	Let $z \in \mathcal Z(A)$ be the unique projection such that $Az$ is completely atomic, and $A(1-z)$ is diffuse. Assume $A$ is not diffuse, so $z \neq 0$. Hence, we can find a family of nontrivial minimal projections $\paP:=\{z_i: i \in I \} \subseteq \mathcal P(\mathcal Z(A))$ and finite dimensional von Neumann algebras $M_i\cong M_{n_i}(\mathbb C)$ for some $n_i$ such that $Az= \oplus_i M_iz_i$. For each $g \in \G$, we denote by $\alpha_g$ the $\ast$-automorphism of $M$ given by $\alpha_g=$ad$(u_g)$. Note that $\alpha_g(Az)=Az$, for all $g \in \G$. Moreover, each $\alpha_g$ leaves the center $\oplus_i \mathbb C z_i$ invariant. Thus $\alpha_g$ leaves the set $\paP$ invariant and hence $\G$ acts on $\paP$ by $g \cdot z_i= \alpha_g(z_i)$. As $\alpha_g$ is trace preserving, and $M$ is a finite factor, there exists a finite subset $\paP_0:=\{z_{i_1}, \ldots, z_{i_n} \} \subseteq \paP$ which is $\G$ invariant. In particular ${\rm span} \paP_0\subset M$ is a finite dimensional $\Gamma$-invariant subspace and by Proposition \ref{nofindiminv} we get that $z_{i_1}=1$ and hence $A= M_1$. Thus $M_1$ is a $\Gamma$ invariant finite dimensional von Neumann algebra and once again Proposition \ref{nofindiminv} implies that $A=\mathbb C1$.
\end{proof}
\vskip 0.05in

\noindent We notice that if $A \subset L(\Gamma)$ is any $\G$-invariant von Neumann subalgebra then so is its center $\mathcal Z = \mathcal Z(A)$. Thus the prior results altogether imply that to understand the structure of $\G$-invariant subalgebras it is imperative to look at the diffuse abelian case. In \cite[Example 3.5]{AJ22} were presented situations of abelian von Neumann subalgebras that do not arise from subgroups. For example if we take any nontrivial wreath product $\Gamma = \Sigma \wr \Lambda$ where $\Sigma$ is abelian then for any nontrivial von Neumann subalgebra $B\subseteq L(\Sigma)$ the infinite tensor product $A:=\overline \otimes_\Lambda  B \subseteq L(\Sigma^{(\Lambda)})\subset L(\G)$ is obviously a $\Gamma$-invariant von Neumann subalgebra which may not arise from any subgroup of $\G$. One can construct more examples of the following type. Let  $\Sigma \lhd \Gamma$ be a normal inclusion where $\G$ is icc and the finite conjugacy radical of $\Sigma$ is infinite and  nonabelian. Then the center $A:= \mathcal Z(L(\Sigma))\subset L(\Sigma) \subset L(\Gamma)$ is an $\Gamma$-invariant  von Neumann subalgebra that does not arise from a subgroup.




\section{Proof of Theorem \ref{b}}

First we show that the presence of combinatorial relations of the type \eqref{eq. polygon} in a group $\Gamma$ is an obstruction to the existence of diffuse abelian $\G$-invariant subalgebra in $L(\Gamma)$. Similar analysis was used to great effect in other classification aspects in von Neumann algebras via deformation/rigidity theory, notably in \cite{Po04,IPP05} and more recently \cite{CIOS22}.  

\begin{thm} \label{thm:acylindricallyinvfactor}
Let $\G$ be a group for which there exists an infinite weakly malnormal subgroup $C < \G$ satisfying the following property:  for every finite subset $F \subset \G \setminus C$ there exists a finite subset $K \subset C$ such that
\begin{equation} \label{eq:acycondition3}
F(C \setminus K)F(C \setminus K) \cap (C \setminus K)F(C \setminus K)F = \emptyset.
\end{equation}
Assume there are commuting von  Neumann subalgebra $A_1,A_2 \subseteq L(\G)$ such that $\G \subset \mathcal N_{L(\G)}(A_1)\cap \mathcal N_{L(\G)}(A_2)$. Then either $A_1$ or $A_2$ is atomic. 
\end{thm}

\begin{proof}
	Assume for the sake of contradiction that both  $A_1$, $A_2$ are diffuse. Henceforth we denote by $M=L(\G)$.
	
	\vskip 0.05in
	\noindent Next we briefly argue that $A_j \npreceq_M L(C)$ for $j=1,2$.  Indeed, as $A_j$ is regular in $M$, if $A_j \preceq_M L(C)$, by Proposition \ref{prop:hpvcor7} we would have that $A_j \preceq_M L(\cap_{i=1} g_iCg_i^{-1})$ for all $g_1,...,g_n \in \G$. However, since $C$ is weakly malnormal there exist $g_1,...,g_k\in \G\setminus C$ such that  $|\cap_{i=1} g_iCg_i^{-1}|< \infty$. Using this we further get that $A_j$ has a nontrivial atomic corner, which contradicts that $A_j$ is diffuse. 
	
	\vskip 0.05in
	\noindent We now prove the following 
	\begin{claim} \label{claim:kaplanskyestimate}
		For every $j=1,2$ and $\varepsilon >0$ there exists $a_j \in \euu(A_j)$, a finite set $F^j_{\varepsilon} \subset \G \setminus C$, and an element  $a^j_{\varepsilon} \in M$ supported on $F^j_{\varepsilon}$ such that $\|a^j_{\varepsilon}\|_{\infty} \leq 2$, and $\|a_j-a^j_{\varepsilon}\|_2< \varepsilon$.
	\end{claim}
	
	\vskip 0.05in
\noindent \textit{Proof of Claim \ref{claim:kaplanskyestimate}:} Fix $j=1,2$ and $\varepsilon>0$. As $A_j \npreceq_M L(C)$, by Theorem \ref{corner} there is $a_j \in \euu(A)$ such that \begin{equation}\label{ineq1}\|E_{L(C)}(a_j)\|_2< \frac{\varepsilon}{3}.\end{equation} By Kaplansky's Density Theorem, there exists $b_j \in M$ supported on a finite subset $G_j \subset \G$ such that \begin{equation}\label{ineq2}\|a_j-b_j\|_2 < \frac{\varepsilon}{3}, \text{ and } \|b_j\|_{\infty} \leq 1.\end{equation}  
\vskip 0.05in
\noindent Now, let $a^j_{\varepsilon}= b_j-E_{L(C)}(b_j)$, and note that $a^j_{\varepsilon}$ is supported on the finite set $F^j_{\varepsilon}:=G_j \setminus C$. Moreover we can see that 
\begin{equation*}\|a^j_{\varepsilon}\|_{\infty}=\|b_j-E_{L(C)}(b_j)\|_{\infty} \leq \|b_j\|_{\infty}+ \|E_{L(C)}(b_j)\|_{\infty} \leq 2. \end{equation*}
Finally, using triangle inequality together with other basic estimates and \eqref{ineq1}-\eqref{ineq2} we see that 
\begin{align*}
\|a_j-a^j_{\varepsilon}\|_2 &= \|a_j-(b_j-E_{L(C)}(b_j))\|_2 \leq \|a_j-b_j\|_2 + \|E_{L(C)}(b_j)\|_2 \\ &\leq \|a_j-b_j\|_2 + \|E_{L(C)}(a_j)\|_2+ \|E_{L(C)}(a_j-b_j)\|_2 \\
& \leq 2\|a_j-b_j\|_2 + \|E_{L(C)}(a_j)\|_2 \leq 3\cdot \frac{\varepsilon}{3}= \varepsilon, 
\end{align*} which finishes the proof of the claim. $\hfill\blacksquare$

\vskip 0.15in

\noindent Fix $j=1,2$ and $\varepsilon >0$, let $a_j \in \euu(A_j)$, $F^j_{\varepsilon} \subset \G \setminus C$ and $a^j_{\varepsilon}$ as in the statement of  Claim~\ref{claim:kaplanskyestimate}. Let $F_\varepsilon=  F^1_\varepsilon \cup F^2_\varepsilon$.  Since $A_1$ and $A_2$ are commuting $\G$ invariant von Neumann subalgebras, we have $u_ga_1 u_g^*a_2=a_2u_ga_1u_g^*$ for all $g \in \G$. Thus using basic estimates we get
\begin{align*}
1&=\|u_ga_1u_g^*a_2\|_2^2= |\langle u_ga_1u_g^*a_2, u_ga_1u_g^*a_2 \rangle|= |\langle u_ga_1u_g^*a_1, a_2u_ga_1u_g^*\rangle  | 
\\
&\leq |\langle u_g(a_1-a^1_{\varepsilon})u_g^*a_2, a_2u_ga_1u_g^* \rangle | + |\langle u_ga^1_{\varepsilon}u_g^*a_2, a_2u_ga_1u_g^* \rangle|  \leq \|a_1-a^1_{\varepsilon}\|_2 + |\langle u_ga^1_{\varepsilon}u_g^*a_2, a_2u_ga_1u_g^* \rangle|\\&
\leq \varepsilon + |\langle u_ga^1_{\varepsilon}u_g^*(a_2-a^2_{\varepsilon}), a_2u_ga_1u_g^* \rangle |+ |\langle u_ga^1_{\varepsilon}u_g^*a^2_{\varepsilon}, a_2u_ga_1u_g^* \rangle| \\
&\leq \varepsilon + \|a^1_{\varepsilon}\|_{\infty}\|a_2-a^2_{\varepsilon}\|_2 + |\langle u_g a^1_{\varepsilon}u_g^*a^2_{\varepsilon}, au_gau_g^* \rangle|
\\
&\leq \varepsilon(1+ \|a^1_{\varepsilon}\|_{\infty}) + |\langle u_ga^1_{\varepsilon}u_g^*a^2_{\varepsilon}, (a_2-a^2_{\varepsilon})u_ga_1u_g^* \rangle |
+ |\langle u_ga^1_{\varepsilon}u_g^*a^2_{\varepsilon}, a^2_{\varepsilon}u_ga_1u_g^*|\\
& \leq \varepsilon(1+\|a^1_{\varepsilon}\|_{\infty}+\|a^1_{\varepsilon}\|_{\infty} \|a^2_\varepsilon\|_\infty) + 
|\langle u_ga^1_{\varepsilon}u_g^* a^2_{\varepsilon},a^2_{\varepsilon}u_g(a_1-a^1_{\varepsilon})u_g^* \rangle | +|\langle u_g a^1_{\varepsilon}u_g^* a^2_{\varepsilon}, a^2_{\varepsilon}u_ga^1_{\varepsilon}u_g^* \rangle|\\
& \leq \varepsilon(1+\|a_{\varepsilon}\|_{\infty}+ \|a^1_{\varepsilon}\|_{\infty} \|a^2_\varepsilon\|_\infty+ \|a^1_{\varepsilon}\|_{\infty} \|a^2_\varepsilon\|^2_\infty)
+ |\langle u_g a^1_{\varepsilon}u_g^* a^2_{\varepsilon}, a^2_{\varepsilon}u_ga^1_{\varepsilon}u_g^* \rangle|.
\end{align*}
\vskip 0.05in

\noindent Since $\|a^j_\varepsilon\|_\infty\leq 2$ this further implies that for all $g\in \Gamma$ we have
\begin{equation} \label{eq:abelianconditionestimate}
1 \leq 15 \varepsilon + |\langle u_g a^1_{\varepsilon}u_g^* a^2_{\varepsilon}, a^2_{\varepsilon}u_ga^1_{\varepsilon}u_g^* \rangle|.
\end{equation}
Now, let $K \subseteq C$ be the finite subset corresponding to $F_\varepsilon$ that satisfies condition~\eqref{eq:acycondition3}. Since $C$ is infinite, there exists $g \in C \setminus K$ such that $g^{-1} \in C \setminus K$, and therefore by condition~\eqref{eq:acycondition3} we have \begin{equation}
F_{\varepsilon}(C\setminus K)F_{\varepsilon}(C\setminus K) \cap (C\setminus K)F_{\varepsilon}(C\setminus K)F_{\varepsilon}= \emptyset.\end{equation}
\vskip 0.05in 
\noindent However, this implies that $\langle u_g a^1_{\varepsilon}u_g^* a^2_{\varepsilon}, a^2_{\varepsilon}u_ga^1_{\varepsilon}u_g^* \rangle=0$. Using this in inequality~\eqref{eq:abelianconditionestimate} we get $1 \leq 15\varepsilon$, which is a contradiction for $\varepsilon$ sufficiently small. This finishes the proof.
\end{proof}

\vskip 0.05in
\noindent \textit{Proof of Theorem \ref{b}}. By Theorem~\ref{thm:acylindricallyinvfactor} we get that $\mathcal Z(B)$ is atomic. By Corollary~\ref{thm:invariantabelian} we get that $\mathcal Z(B)=\mathbb C$, and hence $B$ is a factor.
Using Theorem Theorem~\ref{thm: invariant factors} one can find a normal subgroup $\Sigma \lhd \Gamma$ such that $B\vee (B'\cap L(\Sigma))=L(\Sigma)$. As $\Sigma\lhd \Gamma$ is normal  it follows that its FC-center $\Sigma^{fc} \lhd \Gamma$ is an amenable normal subgroup. Since $\Gamma$ is icc acylindrically hyperbolic it follows that $\Sigma^{fc}=1$. In particular, $\Sigma$ is an icc group and hence $B$ and $B'\cap L(\Sigma)$ are factors. Thus using Theorem~\ref{thm:acylindricallyinvfactor} we get that either $B \cong M_n(\mathbb C)$ or $B'\cap L(\Sigma)\cong M_n(\mathbb C)$ for some $n\in \mathbb N$. Moreover, both $B$ and $B'\cap L(\Sigma)$ are invariant under the conjucagy action of $\Gamma$. However, as $\Gamma$ is icc, Proposition \ref{nofindiminv} implies that $n=1$ which yields the desired conclusion.$\hfill\square$

    \section{Proof of Theorem \ref{a}}
In this section we will investigate $\Gamma$-invariant subalgebras in II$_1$ factors associated with groups satisfying representations-valued cohomological type of negative curvature. Specifically, we will be using deformation/rigidity arguments tailored to the analysis of arrays/quasicocycles from \cite{Pet09, Pet09b, Si10,CP10,Va10,CS11,CSU11,CSU13,CKP15} to prove Theorem \ref{a} and other related results. 
\vskip 0.07in 
\noindent Our first theorem shows that in such group factors if there are $\Gamma$-invariant von Neumann subalgebras that may not be implemented by normal subgroups they have to be necessarily amenable and have diffuse centers.    

\begin{thm} \label{thm:nononameninv}
Let $\G$ be an icc exact group satisfying one of the following conditions:
	 \begin{enumerate}
	 \item [1)] $\G$ admits an unbounded quasi-cocycle into an weakly-$\ell^2$, mixing representation; 
	\item [2)] $\G$ admits a proper array into a weakly-$\ell^2$ representation.
	  
	\end{enumerate}
	Assume $N \subseteq L(\G)$ is a nonamenable von Neumann subalgebra, such that $\G \subseteq \mathcal N_{L(\G)}(N) $. Then, there exists a normal subgroup $\Sg \lhd \G$ such that $N=L(\G)$.
\end{thm}

\begin{proof} First we show  $N$ is a factor. Let $\mathcal Z= \mathcal Z(N)$ be the center of $N$. As $u_g N u_g^*=N$ for all $g \in \G$, we also have $u_g \mathcal Z u_g^*=\mathcal Z $, for all $g \in \G$. 
\vskip 0.05in

By Corollary \ref{thm:invariantabelian}, we may now assume by contradiction that $\mathcal Z$ is diffuse. Let $q: \G \rightarrow \mathcal H_{\pi}$ be an unbounded quasi-cocycle satisfying condition 1) or a proper array satisfying condition 2). Let $M= L(\G)$. Following Section \ref{sec:prelgauss} let $M\subset \tilde M= L^\infty(X^\pi)\rtimes \Gamma$ be the Gaussian extension and let $V_t: L^2(\tilde M) \rightarrow L^2(\tilde M)$ be the Gaussian deformation corresponding to $q$. Let $e_M: \tilde M \rightarrow M$ be the orthogonal projection. 
	\vskip 0.05in
	\noindent Since $N$ is nonamenable, there is $0 \neq p \in \mathcal Z$ such that $N p$ has no amenable direct summand. Thus applying a version of Popa's spectral gap argument (see part d) in Theorem \ref{propgaussdef}) we get 
	\begin{equation} \label{eq:gaussianunitary}
	\lim\limits_{t \rightarrow 0}\left(\underset {a \in (\mathcal Z)_1}{\rm sup}\|e_M^{\perp} \circ V_t(ap)\|_2\right) = 0.
	\end{equation}
Equation~\eqref{eq:gaussianunitary} and  \cite[Proposition 6.4]{CKP15} imply that for every $g \in \G$ we have $	\lim\limits_{t \rightarrow 0}(\underset {a \in (\mathcal Z)_1}{\rm sup}\|e_M^{\perp} \circ V_t(u_gapu_g^*)\|_2 )= 0.$ Since $u_g(\mathcal Z)_1 u_g^*= (\mathcal Z)_1$, we further have
\begin{equation} \label{eq:gaussianunitaryupgrade}
\lim\limits_{t \rightarrow 0}\left(\underset {a \in (\mathcal Z)_1}{\rm sup}\|e_M^{\perp} \circ V_t(au_gpu_g^*)\|_2\right) = 0, \text{ for all } g \in \G.
\end{equation}
Next we prove the following
\begin{claim} \label{claim:fromcornertoall}
 $\lim\limits_{t \rightarrow 0}\left(\underset {a \in (\mathcal Z)_1}{\rm sup}\|e_M^{\perp} \circ V_t(a)\|_2\right)  = 0$.
 \end{claim}
\noindent \textit{Proof of Claim \ref{claim:fromcornertoall}}. Fix $\varepsilon >0$. As $L(\G)$ is a factor, by a  standard convexity argument one can find $g_1, \ldots, g_n \in \G$, and $\mu_1, \ldots, \mu_n >0$, with $\sum_{i=1}^n \mu_i =1$, such that
\begin{equation} \label{eq:weakdixmier}
\| \sum\limits_{i=1}^n \mu_i u_{g_i}pu_{g_i}^*- \tau(p)1 \|_2 < \dfrac{\varepsilon \tau(p)}{2}.
\end{equation}
By \eqref{eq:gaussianunitaryupgrade} on can find $t_{\varepsilon}>0$ such that for all $0 < |t| \leq t_{\varepsilon}$ we have
\begin{equation} \label{eq:gaussianunitaryupgrade2}
\|e_M^{\perp} \circ V_t(au_gpu_g^*)\|_2 < \dfrac{\varepsilon \tau(p)}{2}, \text{ for all } i \in \{1, \ldots, n \}, \text{ } a \in (\mathcal Z)_1.
\end{equation}
Fix $a \in (\mathcal Z)_1$ and $0< |t| \leq t_{\varepsilon}$. Then triangle inequality, and  equations~\eqref{eq:weakdixmier}, ~\eqref{eq:gaussianunitaryupgrade2} show that
\begin{align*}
\tau(p)\|e_M^{\perp} \circ V_t (a) \|_2 &= \|e_M^{\perp} \circ V_t (a \tau(p)) \|_2 \leq \dfrac{\varepsilon \tau(p)}{2}+ \|e_M^{\perp} \circ V_t (a (\sum\limits_{i=1}^n \mu_i u_{g_i}pu_{g_i}^*)) \|_2 \\
& \leq \dfrac{\varepsilon \tau(p)}{2} + \sum\limits_{i=1}^n \mu_i \| e_M^{\perp} \circ V_t (au_{g_i}pu_{g_i}^*) \|_2 \leq \dfrac{\varepsilon \tau(p)}{2}+ \dfrac{\varepsilon \tau(p)}{2}= \varepsilon \tau(p).
\end{align*}
Hence, $\|e_M^{\perp} \circ V_t(a)\|_2 \leq \varepsilon$. As $a \in (\mathcal Z)_1$, $0 < |t| \leq t_{\varepsilon}$ were arbitrary, the claim follows. \hfill $\blacksquare$ 
\vskip 0.05in

\noindent Notice that the prior claim and transversality property imply that $V_t\rightarrow I$ as $t\rightarrow 0$ uniformly on $(\mathcal Z)_1$. Thus if we where on case 1) Theorem \ref{atomiccorner} already  leads to a contradiction. So assume we are in case 2).

\noindent Let $C>0$. We denote by $\mathfrak{B}_C=\{g \in \G: \|q(g)\| \leq C \}$, and by $\mathcal P_{\mathfrak{B}_C}$ the orthogonal projection onto the Hilbert subspace of $\mathfrak{B}_C$ inside $\ell^2(\G)$. Since $V_t\rightarrow I$ as $t\rightarrow 0$ uniformly on $(\mathcal Z)_1$, by item e) in Theorem \ref{propgaussdef},  for every $\varepsilon >0$, one can find $C>0$ such that
\begin{equation} \label{eq:closeprojectionofelements}
\|a- \mathcal P_{\mathfrak{B}_C}(a)\|_2 \leq \varepsilon, \text{ for all } a \in (\mathcal Z)_1.
\end{equation}

\noindent Since $q$ is proper, $\mathfrak B_C$ are finite, and the prior inequality already contradicts that $\mathcal Z$ is diffuse.

\vskip 0.06in

\noindent Using Theorem \ref{thm: invariant factors} one can find $\Sigma\lhd \Gamma$ a normal subgroup such that $N\vee (N'\cap L(\Sigma))=L(\Sigma)$. As $N$ has no amenable direct summand then Popa's spectral gap argument implies that

 $$\lim\limits_{t \rightarrow 0}\left(\underset {a \in (N'\cap L(\Sigma))_1}{\rm sup}\|e_M^{\perp} \circ V_t(a)\|_2\right)  = 0.$$

\noindent By transversality, this implies that $V_t\rightarrow I$, as $t\rightarrow 0$, uniformly on $(N'\cap L(\Sigma))_1$. Since  $N'\cap L(\Sigma)$ is $\Gamma$ -invariant, Theorem \ref{atomiccorner} implies that  $N'\cap L(\Sigma)$ is not diffuse. Thus it admits a nontrivial atomic corner that is also invariant under the conjugation action by $\Gamma$. Then using Corollary \ref{thm:invariantabelian} we get that $N'\cap L(\Gamma)=\mathbb C1$ and hence $N=L(\Sigma)$, as desired. 
\end{proof}

\noindent In view of the prior result and Theorem \ref{thm: invariant factors} we are left to analyze amenable $\Gamma$-invariant subalgebras and more specifically the case of the abelian ones. Before proceeding to this part of the proof we need to record the following lemma which is also of independent interest.
\begin{lem}\label{inneramenablepiece} Let $A \subseteq L(\Gamma)=:M$ be an abelian von Neumann subalgebra that is $\Gamma$-invariant. Let $g\in \Gamma$ such that $E_A(u_g)\neq 0$. If we denote by $\Omega= vC_\Gamma(\langle g\rangle )\leqslant \Gamma$ one can find a projection $0\neq z_g\in A'\cap \mathcal Z(L(\Omega))$ such that $A z_g\subseteq L(\Omega)$.

\end{lem}

\begin{proof}
    Fix $g \in \G$ with $E_A(u_g) \neq 0$. We denote by $\alpha_g$ the automorphism of $M$ given by  ${\rm ad}(u_g)$. Note that $\alpha_g$ restricts to an automorphism of $A$. We thus have
	$\alpha_g(x)u_g=u_gx$ for all $x \in A$. Applying the conditional expectation $E_A$ we also get 
	$\alpha_g(x)E_{A}(u_g)=E_{A}(u_g)x$ for all $x \in A$. These two relations combined imply that $E_A(u_g)u_g^* \in A' \cap M$. Let $e_g \in A' \cap M$ be such that $E_A(u_g)=e_gu_g$. As $E_A(u_g) \neq 0$, we have $e_g \neq 0$. Also notice  that $E_A(e_g)=e_g e_g^*=E_A(u_g)E_A(u_g^*)\geq 0$.  
	We denote by $f=\text{sup}(E_A(e_g))= \text{sup}(e_g e_g^*)= \text{sup}(|e_g^*|)\in A $ the support projection. Since $e_g u_g \in A \subseteq A'\cap M$ then  $ae_g u_g =e_g u_g a$ for all $ a\in A$. Therefore $a e_g = \alpha_g(a) e_g$ for all $a\in A $. This implies that $e e_ge_g^* = \alpha_g(a) e_ge_g^*$ and hence $a f = \alpha_g(a) f$ for all $a\in A$; in particular, $\alpha_g$ is identity on $Af$. This further entails that 
	 $Af \subseteq L(\langle g \rangle)' \cap L(\G) \subseteq L(vC_{\G}(\langle g \rangle))$.  \vskip 0.05in
	 \noindent  To this end denote by $\Omega:= vC_{\G}(\langle g \rangle) \leqslant \G$. As $Af \subseteq L(\Omega)$, we get that $u_hAfu_h^* \subseteq L(h \Omega h^{-1})= L(\Omega)$, for all $h \in \Omega$. As $A$ is $\Omega$-invariant, this further yields
	 $A u_h fu_h^* \subseteq L(\Omega)$ for all $h \in \Omega$. Hence, letting $z_g= \vee_{h\in \Omega} u_hfu_h^*\in \mathcal Z(L(\Omega))$ we further get $A z_g \subseteq L(\Omega)$.
\end{proof}

\noindent With these preparations at hand we are ready to derive the proof of Theorem \ref{a} under the assumptions of item 2).

\begin{thm}\label{isrquasi}
Let $\G$ be a torsion free, exact group that admits an unbounded quasi-cocycle into a mixing, weakly-$\ell^2$ representation. Then  $L(\G)$ satisfies the ISR property.

\end{thm}

\begin{proof}Let $N\subseteq L(\Gamma)$ be a $\Gamma$-invariant von Neumann subalgebra. Let $A =\mathcal Z(N)$ and note it is $\G$-invariant as well. Next we show that $A=\mathbb C1$. Suppose by contradiction $A\neq \mathbb C 1$. By Corollary \ref{thm:invariantabelian}  $A$ must be diffuse.
	Let $1\neq \Sg \lhd \G$ be the (infinite) smallest normal subgroup of $\G$ such that $A \subseteq L(\Sg)$. The existence of $\Sg$ is guaranteed by Zorn's lemma together with the fact that $A$ is normalized by $\G$.  
	\vskip 0.07in
\noindent	Observe that \cite[Lemma 3.3, Theorem 3.4]{T09} imply that $\Sg$ is nonamenable. Next we briefly argue that $\Sg$ is icc. Using the assumptions and \cite[Theorem 7.1]{CKP15} we must have that $FC(\Sg)$ is finite. Moreover, since $\Sigma\lhd \Gamma$ it follows that $FC(\Sg)\lhd \Gamma$ is also normal. As $\Gamma$ is icc we conclude that $FC(\Sg)=1$ and hence $\Sg$ is icc. In conclusion $\Sg\lhd \G$ is a normal icc subgroup. Therefore, if $q$ is an unbounded quasicocycle into a mixing weakly-$\ell^2$ representation of $\G$ then its restriction $q_{|_\Sg}$ to $\Sg$ is also an unbounded quasicocycle into a mixing, weakly-$\ell^2$ representation of $\Sg$ (see \cite[Proposition 4.4 d)]{CKP15}).  Therefore, using the prior paragraph we can assume without any loss of generality that $A\not\subset L(\Sigma)$ for any proper  subgroup $\Sigma < \Gamma$.  
	\vskip 0.07in 
\noindent Since $A\neq \mathbb C 1$ then one can find $1\neq g\in \Gamma$ such that if we denote by $\Omega =vC_\Gamma (\langle g \rangle )$ then one can find $0\neq z\in A'\cap \mathcal Z(L(\Omega))$ such that $A z \subseteq L(\Omega)$. As $\Gamma$ is torsion free then $\langle g \rangle<\Gamma$ is an infinite cyclic subgroup. As $\Omega= vC_{\G}(\langle g \rangle)$, from definitions there is an increasing  sequence of finitely generated subgroups $\cdots \Omega_n\leqslant \Omega _{n+1}\leqslant \cdots \leqslant \Omega$ such that $\cup_n \Omega_n =\Omega$ and the centralizers $C_{\langle g \rangle }(\Omega _n )\leqslant \langle g\rangle$ has finite index for all $n$.  Altogether these imply that $L(\Omega)$ has property \emph{Gamma} of Murray and von Neumann. Now assume that $\Omega$ is non-amenable. Then using \cite[Theorem 3.1]{CSU13} we must have that the quasi-cocycle $q$ is bounded on $\Omega$. In particular, the corresponding deformation $e^\perp_M \circ V_t \ra 0$ on the unit ball of $L(\Omega)$. Thus $e^\perp_M \circ V_t \ra 0$ on the unit ball of $Af$. Then arguing exactly as in the proof of Theorem~\ref{thm:invariantabelian}, we reach a contradiction.
	 Hence, $\Omega$ must be amenable. As $Az \subseteq L(\Omega)$, and as $\G$ normalizes $A$, we get $A \preceq_{L(\G)}^s L(\Omega)$. Using Proposition~\ref{prop:hpvcor7} this further implies that $|\Omega \cap g \Omega g^{-1}|= \infty$, as $A$ is diffuse. This however implies that $\Omega$ is an s-normal subgroup of $\G$, which is impossible by \cite[Lemma 3.3, Theorem 3.4]{T09}.
 \vskip 0.05in 
 \noindent Thus $N $ is a factor. If it is nonamenable the conclusion follows from Theorem~\ref{thm:nononameninv}. So assume $N$ is amenable. By Theorem \ref{thm: invariant factors} one can find a normal subgroup $\Sg\lhd \Gamma $ such that $N\vee (N'\cap L(\Sg))=L(\Sg) $. Since $\Gamma$ does not have amenable normal subgroup, it follows that the FC-center of $\Sigma$ is trivial and hence $\Sigma$ is icc. Moreover, we have that $\Sigma$ is a non-amenable. Recall that two commuting amenable subalgebras generate an amenable subalgebra in a $\rm II_1$ factor. Thus $N'\cap L(\Sigma)$ is nonamenable  and by Theorem~\ref{thm:nononameninv}, $N'\cap L(\Sigma)=L(\Omega)$ for some icc non-amenable subgroup $\Omega<\Sigma$. Using factoriality we also have that $N= L(\Sigma )\cap (N'\cap L(\Sigma))'= L(\Sigma )\cap L(\Omega)'\subseteq L(vC_\Sigma(\Omega))$.  However, using Ge's splitting theorem we have that $L(vC_\Sigma(\Omega))= N\bar\otimes P $ where $P= L(vC_\Sigma(\Omega))\cap L(\Omega)$. However, since $vC_\Sigma(\Omega)\cap \Omega $ is the FC-center of $\Omega$ and $\Omega$ is icc we get $vC_\Sigma(\Omega)\cap \Omega =1$ and hence $P= \mathbb C 1$. In conclusion, $N= L(vC_{\Sigma}(\Omega))$; in fact, it is easy to see we have $vC_\Sigma(\Omega)=C_\Sigma(\Omega)$. In particular, this implies the desired conclusion. \end{proof}

\vskip 0.07in 
\noindent \textbf{Remark:} The argument in the last paragraph of the above proof in fact demonstrates that there are no $\Gamma$-invariant amenable subfactors inside $L(\G)$. Indeed, as $\G$ is $C^*$-simple, and $\Sg \lhd \G$, $\Sg$ is $C^*$-simple. Hence so are $\Omega$ and $C_\Sigma(\Omega)$. As $N$ is amenable, this forces $C_\Sigma(\Omega)$ to be amenable, which in turn forces $C_\Sigma(\Omega)$ to be trivial, as the amenable radical of $\Sg$ is trivial.
\vskip 0.07in 
\noindent Now we derive the proof of Theorem \ref{a} under the assumptions of item 1). Our approach follows closely the general strategy developed in the proof of \cite[Lemma 3.1, Theorem 3.2]{CP10} involving analysis of unbounded derivations (see also \cite[Theorem 4.1]{Va10} for the version of this analysis in the context of Gaussian deformations) and we include all the details just for the reader's convenience.

\begin{thm}
Let $\G$ be any icc group that admits an unbounded, non-proper, $1$-cocycle into a mixing representation; in particular, $\G$ can be any non-amenable group $\G$ that has positive first $L^2$-Betti number and admits an infinite amenable subgroup (e.g.\ when $\G$ is torsion free). \\ Then  $L(\G)$ satisfies the ISR property.
\end{thm}

\begin{proof} Let $M= L(\G)$. Let $A =\mathcal Z(N)$ be its center and notice it is $\G$-invariant as well. Assume, by contradiction that $A$ is a diffuse.	\vskip 0.05in 
\noindent  Since $q$ is nonproper, we can find $C>0$ such that $\mathfrak B_C:=\{g \in \G: \|q(g)\| \leq C \}$ is infinite. Fix $(g_n)_n \subset \mathfrak B_C$ an infinite set of group elements. Let $\omega$ be a nonprincipal ultrafilter, and let $E_M: M^{\omega} \rightarrow M$ denote the unique trace preserving conditional from the ultrapower $M^{\omega}$ onto $M$. We denote by $ u^{\omega}=(u_{g_n})_n \in \euu(M^{\omega})$. 
\vskip 0.07in
\noindent Just as in \cite{CP10}, we continue by splitting the proof into two cases that we analyze separately.
    \vskip 0.1in
	\noindent \textbf{Case I:} There exists $a \in \euu(A)$ such that $E_M( u^{\omega}a u^{\omega \ast})= \lim_{n \rightarrow \omega}u_{g_n}au_{g_n}^{\ast}=0$.   
  \vskip 0.07in   
 \noindent Let $c_n = u_{g_n}au_{g_n}^* \in \euu(A)$.   Fix $b \in \euu(A)$. Using that $c_n b=bc_n$ for all $n$ and then  $V_t(xy)=V_t(x)V_t(y)$ for all $x,y$ basic estimates show that for all $n$ we have
	\begin{equation}\label{eq:conv1}\begin{split}
	\|e_M^{\perp} \circ V_t (b)\|_2^2 &= |\langle c_n e_M^{\perp} \circ V_t(b), c_n e_M^{\perp} \circ V_t(b) \rangle | 
	 \\&\leq |\langle c_ne_M^{\perp} \circ V_t(b), e_M^{\perp} \circ V_t(c_n b) \rangle| + \| c_n V_t(b)-V_t(c_n b) \|_2 \\
	 & = |\langle c_n e_M^{\perp} \circ V_t(b)c_n ^*, e_M^{\perp} \circ V_t(b) \rangle|+ \| c_nV_t(b)-V_t(c_n b) \|_2 +\| V_t(b)c_n-V_t(bc_n) \|_2\\
	& \leq |\langle c_n e_M^{\perp} \circ V_t(b)c_n^*, e_M^{\perp} \circ V_t(b) \rangle |+ 2 \| c_n-V_t(c_n ) \|_2 
	 \end{split}	\end{equation} 
    
\noindent Since $g_n\in \mathfrak B_C$ for all $n$ and since $V_t \ra {\rm Id}$ pointwise then the same argument from the proof of part e) in Theorem \ref{propgaussdef} shows that $\sup_n\| V_t(c_n)-c_n\|_2\ra 0$. as $t\ra 0$. Moreover, since $\lim_{n \rightarrow \omega} \|E_M(c_n)\|_2 =0$, mixingness of the representation (see for instance \cite[Lemma 2.5]{CP10} or the earlier results \cite{PS12,Pet09b}) shows that $\lim_n |\langle c_n e_M^{\perp} \circ V_t(b)c_n^*, e_M^{\perp} \circ V_t(b) \rangle |=0$. Using these and taking limit over $n$ in equation~\eqref{eq:conv1} and we get that $e_M^{\perp} \circ V_t \rightarrow 0$ uniformly on $(A)_1$.  This however leads to a contradiction using the last part of the proof of Theorem \ref{thm:nononameninv}.

\vskip 0.1in

\noindent \textbf{Case II:} There are $1>C>0$, $a \in \euu(A)$ so that 
$\|E_M(u^{\omega}au^{\omega \ast})\|_2=\|\lim_{n\ra \omega} u_{g_n}a u_{g_n}^*\|_2= C>0$. 
   \vskip 0.07in  
\noindent Let $c_n = u_{g_n}au_{g_n}^* \in \euu(A)$ and denote by $c^{\omega}=(c_n)_n=u^{\omega}au^{\omega \ast}\in A^\omega$.   As in the proof of \cite[Lemma 3.1, Case II]{CP10}, let $u \in \euu(A'\cap M)$ such that $E_M(c^{\omega})=u|E_M(c^{\omega})|\in A'\cap M$.  Also, let $p \in \mathcal P(A'\cap M)$ be the spectral projection of $|E_M(c^{\omega})|$ corresponding to $[0, 1-\frac{C}{2}]$. Then we have that
$\| |E_M(c^{\omega})| \cdot (1-p)\|_2 \geq (1-\frac{C}{2})\| 1-p\|_2 $. Arguing as in proof of \cite[Lemma 3.1, Case II]{CP10} this further implies
$\|p\|_2 \geq 1-\|1-p\|_2 \geq 1-\frac{2}{2-C}\| E_M(c^{\omega})\|_2 \geq 1-\frac{2}{2-C}(1-C)=\frac{C}{2-C}$. Hence, $\tau(p) = \|p\|_2\geq \frac{C}{2-C}$; in particular $p\neq 0$.  Moreover, if $y:=E_M(u^* c^{\omega})=|E_M(c^\omega)| \in M$, one can also check that

 \begin{equation}\label{snormbound}\| y p   \|_{\infty} \leq 1-\frac{C}{2}.\end{equation} 
 
 \noindent Let   $s_n:=u^*c_n-y$ for all $n$. 
 
 \noindent Now fix $b \in \euu(A)$. Using $c_n b =bc_n$ for all $n$ together with basic calculations as in the previous case and \eqref{snormbound} we get
   
\begin{align*}
\|pe_M^{\perp} \circ V_t(b)\|_2^2 &= |\langle u^* c_n pe_M^{\perp} \circ V_t(b), u^* c_n pe_M^{\perp} \circ V_t(b) \rangle | \\
&\leq |\langle s_n pe_M^{\perp} \circ V_t(b), u^* c_n pe_M^{\perp} \circ V_t(b) \rangle| + |\langle y pe_M^{\perp} \circ V_t(b), u^* c_n pe_M^{\perp} \circ V_t(b) \rangle| \\
& \leq |\langle us_n pe_M^{\perp} \circ V_t(b), p  e_M^{\perp} \circ V_t( c_nb) \rangle | +\| c_n V_t(b)-V_t(c_n b) \|_2+ |\langle ype_M^{\perp} \circ V_t(b), u^* c_n pe_M^{\perp} \circ V_t(b) \rangle | \\
& \leq |\langle us_n pe_M^{\perp} \circ V_t(b)c_n^*, p  e_M^{\perp} \circ V_t( b) \rangle |+ \| V_t(b)c_n-V_t(bc_n ) \|_2\\&+\| c_n V_t(b)-V_t(c_n b) \|_2+ |\langle ype_M^{\perp} \circ V_t(b), u^* c_n pe_M^{\perp} \circ V_t(b) \rangle | \\
& \leq |\langle us_n pe_M^{\perp} \circ V_t(b)c_n^*, p  e_M^{\perp} \circ V_t( b) \rangle |+ 2\| c_n-V_t(c_n ) \|_2+ |\langle ype_M^{\perp} \circ V_t(b), u^* c_n pe_M^{\perp} \circ V_t(b) \rangle | \\
& \leq |\langle us_n pe_M^{\perp} \circ V_t(b)c_n^*, p  e_M^{\perp} \circ V_t( b) \rangle |+ 2\| c_n-V_t(c_n ) \|_2\\&+ |\langle ype_M^{\perp} \circ V_t(b), s_n pe_M^{\perp} \circ V_t(b) \rangle |+ \|ype_M^{\perp} \circ V_t(b)\|^2_2\\
& \leq |\langle us_n pe_M^{\perp} \circ V_t(b)c_n^*, p  e_M^{\perp} \circ V_t( b) \rangle |+ 2\| c_n-V_t(c_n ) \|_2\\&+ |\langle ype_M^{\perp} \circ V_t(b), s_n pe_M^{\perp} \circ V_t(b) \rangle |+ \left(1-\frac{C}{2}\right )\|pe_M^{\perp} \circ V_t(b)\|^2_2
\end{align*}

\noindent This further shows that 
\begin{equation}\label{eq:conv2}
     \|pe_M^{\perp} \circ V_t(b)\|_2^2\leq \frac{2}{C}( |\langle us_n pe_M^{\perp} \circ V_t(b)c_n^*, p  e_M^{\perp} \circ V_t( b) \rangle |+\langle ype_M^{\perp} \circ V_t(b), s_n pe_M^{\perp} \circ V_t(b) \rangle |+ 2\| c_n-V_t(c_n ) \|_2)
\end{equation}
  \noindent  Once again  we have $\sup_n\| V_t(c_n)-c_n\|_2\ra 0$, as $t\ra 0$. Since  $\lim_n\|E_M( us_n p)\|_2 = \lim_n\|E_M( s_n p)\|_2=0$ the mixingness of the representation shows that $\lim_n |\langle us_n pe_M^{\perp} \circ V_t(b)c_n^*, p  e_M^{\perp} \circ V_t( b) \rangle |=0$ and $\lim_n |\langle ype_M^{\perp} \circ V_t(b), s_n pe_M^{\perp} \circ V_t(b) \rangle |= 0$. Thus taking the limit over $n$ in equation \eqref{eq:conv2} we get that \begin{equation*}\lim_{t\ra 0} \left(\sup_{b\in \mathcal U(A)}\|p e_M^\perp \circ V_t(b)\|_2\right )
=0.\end{equation*}

\noindent Since $V_t(p)V_t(b)= V_t(pb)$ for all $b$ and $\|p-V_t(p)\|\ra 0$ as $t\ra 0$  this further implies that \begin{equation*}\lim_{t\ra 0} \left(\sup_{b\in \mathcal U(A)}\| e_M^\perp \circ V_t(p b)\|_2\right )
=0.\end{equation*}

\noindent Proceeding as in the proof of Claim \ref{claim:fromcornertoall} this further implies that $\lim_{t\ra 0} \left(\sup_{b\in \mathcal U(A)}\| e_M^\perp \circ V_t(p b)\|_2\right )
=0$. Finally, using the last part of the proof of Theorem \ref{thm:nononameninv}, this leads to a contradiction.  
\vskip 0.06in 
\noindent Thus $N$ is a factor and by Theorem \ref{thm: invariant factors} one can find a normal subgroup $\Sg\lhd \Gamma $ such that $N\vee (N'\cap L(\Sg))=L(\Sg) $. Since $\Gamma$ does not have amenable normal subgroup, it follows that the FC-center of $\Sigma$ is trivial and hence $\Sigma$ is icc. Moreover, we have that $\Sigma$ is a non-amenable and hence at least $N$ or $P=N'\cap L(\Sigma)$ is nonamenable factor. Assume $P$ is nonamenable.  Popa's spectral gap argument implies that

 $$\lim\limits_{t \rightarrow 0}\left(\underset {a \in (N)_1}{\rm sup}\|e_M^{\perp} \circ V_t(a)\|_2\right)  = 0.$$

\noindent By transversality, this implies that $V_t\rightarrow I$, as $t\rightarrow 0$, uniformly on $(N)_1$. Since  $N$ is $\Gamma$ -invariant, Theorem \ref{atomiccorner} implies that  $N$ is not diffuse. Thus it admits a nontrivial atomic corner that is also invariant under the conjugation action by $\Gamma$. Then using Corollary \ref{thm:invariantabelian} we get that $N=\mathbb C1$, as desired. The other case also implies through a similar argument that $N'\cap L(\Sigma)= \mathbb C 1$ and hence $N= L(\Sigma)$, which concludes the proof.\end{proof}

\subsection{Applications to invariant subalgebras of reduced group $C^*$-agebras}
In this subsection we collect together several immediate consequences of our main results to the study of invariant subalgebras or reduced group $C^*$-algebras. To derive our results we first notice the following elementary result.
\begin{lem} \label{lem: vN to C-star}
	Let $\Sigma <\G$ be a countable discrete groups. If $ A \subseteq C^*_r(\G) $ is a $C^*$-subalgebra such that $ A \subseteq L(\Sg)$, then $ A \subseteq C^*_r(\Sg)$.
\end{lem}

\begin{proof}
	Fix $a \in A$ and let $a_n \in \mc[\G]$ be a sequence such that $\|a-a_n\|_{\infty} \rightarrow 0$. Since $A \subseteq L(\Sg)$, applying $E_{L(\Sg)}$ we obtain
	
	\begin{equation*}
\|a- E_{L(\Sg)}(a_n)\|_{\infty}=\|E_{L(\Sg)}(a-a_n)\|_{\infty}  \leq \|a-a_n\|_{\infty} \rightarrow 0.
	\end{equation*}
	
\noindent As $a_n \in \mc[\G]$, clearly $E_{L(\Sg)}(a_n) \in C^*_r(\Sg)$ and hence $a \in C^*_r(\Sg)$.\end{proof}

\begin{cor}\label{C*-app1} Let $\G$ be any icc group that satisfies any of the hypotheses of Theorems \ref{b}, \ref{a}, or Theorem \ref{isrquasi}. Let $A \subseteq C^*_r(\G)$ be any $\Gamma$-invariant $C^*$-subalgebra. Then one can find a normal subgroup $\Sigma \lhd \G$ such that $A \subseteq C^*_r(\Sigma)$ and $A''= L(\Sigma)$.  
    
\end{cor}

\begin{proof} Let $B = A''\subseteq L(\G)$, be the von Neumann algebra generated by $A$ inside $C^*_r(\G)''=L(\G)$. Since $A$ is $\Gamma$-invariant one can see that $B$ is a $\Gamma$-invariant von Neumann subalgebra of $L(\G)$. Using the conclusions of Theorems \ref{b}, \ref{a}, or Theorem \ref{isrquasi}, one can find a normal subgroup $\Sigma \lhd \G$ such that $B=L(\Sigma)$. Using Lemma \ref{lem: vN to C-star} we also have that $A \subseteq C^*(\Sigma)$, which gives the desired conclusion.\end{proof}

\begin{cor}Let $\G$ be any icc group as in Corollary \ref{C*-app1}. Let $A \subseteq C^*_r(\G)$ be any $\Gamma$-invariant $C^*$-subalgebra with expectation. Then one can find a normal subgroup $\Sigma \lhd \G$ such that $A= L(\Sigma)$.  \end{cor}

\begin{proof}
By Corollary~\ref{C*-app1}, there exists $\Sigma \lhd \G$ such that $A \subseteq C^*_r(\Sigma)$, and $A''= L(\Sigma)$. Moreover, from the proofs of the Theorems \ref{b}, \ref{a}, and \ref{isrquasi} it follows that the center $Z(L(\Sigma))=\mathbb C$. Therefore, $A' \cap C^*_r(\Sigma) \subseteq A' \cap L(\Sigma)= \mathcal Z(L(\Sigma))=\mathbb C$. 
\vskip 0.07in 
\noindent Let $E: C^*_r(\G) \rightarrow A$ be a conditional expectation. (We are not assuming $E$ is trace preserving). Notice that $E$ restricts to a conditional expectation $E:C^*_r(\Sigma) \rightarrow A$. Fix $g \in \Sigma$ and consider the automorphism $\theta(x)={\rm ad} (u_g)$ of  $A$. As $\Sigma$ normalizes $A$, we get that $\theta(x)u_g=u_gx$ for all $x \in A$. Hence, $u_g^*E(u_g)x=xu_g^*E(u_g)$ for all $x \in A$, which implies that $u_g^*E(u_g) \in A' \cap C^*_r(\Sigma)= \mathbb C$. Thus, \begin{equation}\label{proj1} E(u_g)=c_gu_g,\text{ where }c_g \in \mathbb C.\end{equation} In particular, this yields that whenever $c_g \neq 0$, then $u_g \in A$.
\vskip 0.07in

\noindent Let $S=\{g \in \Sigma: E(u_g) \neq 0 \}$ and denote by $\Sigma_0 :=\langle S\rangle< \G$, the subgroup generated by $S$.
By the previous paragraph, we get that $C^*_r(\Sigma_0) \subseteq A\subseteq C^*_r(\Sigma)$. In particular, we have $\Sigma_0\leqslant \Sigma$.
\vskip 0.07in

\noindent  Next we claim that $A=C^*_r(\Sigma_0)$. Towards this fix $\varepsilon >0$ and $a\in A$. As $\mathbb C [\Sigma]$ is norm dense in $C^*_r(\Sigma)$, we can find a finite set $F \subset \Sigma$  such that   $\|a-\sum_{g \in F}a_gu_g\| < \varepsilon$.  As $E$ is norm-decreasing, this estimate together with \eqref{proj1} imply that $\varepsilon> \|a-\sum_{g \in F}a_gu_g\|\geq \|E(a -\sum_{g \in F}a_gu_g)\|= \|a - \sum_{g\in F\cap \Sigma_0} a_g c_g u_g\|$. Since this holds for all $\varepsilon>0$ and $a\in A$ it follows that $\mathbb C [\Sigma_0]$ is norm dense in $A$, which establishes our claim.
\vskip 0.07in
\noindent Finally, since $A''=L(\Sigma)$ and $C^*_r(\Sigma_0)''=L(\Sigma_0)$, we get $\Sigma_0 =\Sigma$, thereby finishing the proof.\end{proof}


\subsection{Final remarks and open problems }

We believe that the ISR property of $L(\G)$ is a condition that is more intrinsic to a specific algebraic structure of $\G$ which happens to be implicit in both cases, when $\G$ is a lattice in a higher rank Lie group or when it satisfies a negative curvature property. Namely, we believe $L(\G)$ satisfies the ISR property whenever $\G$ has trivial amenable radical. More precisely, we conjecture the following

\begin{conjecture} \label{conj} Let $\G$ be an icc group and let
	 $A \subset L(\G)$ be a diffuse abelian von Neumann subalgebra such that $\mathcal N_{L(\G)}(A)\subseteq \G$. Then  one can find an amenable normal subgroup $\Sigma \lhd \G$ such that $A \subseteq L(\Sigma)$.   
\end{conjecture}

\noindent Establishing this conjecture in its full generality seems difficult at this time. However, we propose to investigate the following, seemingly easier intermediate conjecture which is already hinted by our previous results in the case torsion free groups $\G$.

\begin{conjecture} \label{conj} Let $\G$ be an icc group and let
	 $A \subset L(\G)$ be a diffuse abelian von Neumann subalgebra such that $\mathcal N_{L(\G)}(A)\subseteq \G$. Then one can find an amenable, s-normal subgroup $\Sigma < \G$ and a nonzero projection $ z\in A'\cap \mathcal Z(L(\Sigma))$ such that $Az \subseteq L(\Sigma)$.   
\end{conjecture}

\noindent Notice that when $\Gamma$ is torsion free this is satisfied if instead of $\Sigma$ amenable we require $\Sigma$ to have infinite FC-radical (and hence inner amenable, etc).

\section*{Acknowledgments}  The authors would like to thank Prof. Jesse Peterson and Prof. Denis Osin for many stimulating discussions and for their support and encouragement. The second author would also like to thank Prof. Mehrdad Kalantar for his kind words regarding the contents of the paper.


\footnotesize{

}





\vskip 0.1in

\noindent \textsc{Department of Mathematics, The University of Iowa, 14 MacLean Hall, Iowa City, IA 52242, U.S.A.}\\
\email {ionut-chifan@uiowa.edu} \\
\textsc{Department of Mathematics, UC Riverside, Skye Hall, Riverside, CA 92521, U.S.A.}\\
\email{sayan.das@ucr.edu}
\\
\textsc{Mathematical Institute, University of Oxford, Radcliffe Observatory, Andrew Wiles Building, Oxford, OX26GG , U.K.}\\
\email{bin.sun@maths.ox.ac.uk }

\end{document}